\documentclass[11pt]{amsart} 
\usepackage{amsmath,amsfonts,amssymb,amsthm,amscd}
\usepackage{hyperref}
\usepackage{dsfont}
\usepackage{color}
\newtheorem{theorem}{Theorem}[section]
\newtheorem{thm}[theorem]{Theorem}
\newtheorem{prop}[theorem]{Proposition}
\newtheorem{lem}[theorem]{Lemma}

\newtheorem{fact}[theorem]{Fact}
\newtheorem{cor}[theorem]{Corollary}
\newtheorem{proposition}[theorem]{Proposition}
\newtheorem{lemma}[theorem]{Lemma}

\newtheorem{conj}[theorem]{Conjecture}

\theoremstyle{definition}

\newtheorem{definition}[theorem]{Definition}

\theoremstyle{remark}
\newtheorem{rem}[theorem]{Remark}
\newtheorem{remark}[theorem]{Remark}

\newcommand{\N}{\mathbb{N}}

\newcommand{\R}{\mathbb{R}}

\newcommand{\zar}[2]{\overline{#1}^{{#2}}}

\newcommand{\ff}{\varphi}                   
\newcommand{\al}{\alpha}                    
\newcommand{\be}{\beta}                     

\newcommand{\sub}{\subseteq}    
\newcommand{\st}{\;|\;}

\global\long\def\g{\mathfrak{g}}
\newcommand{\gl}{\mathfrak{gl}}
\global\long\def\h{\mathfrak{h}}

\global\long\def\n{\mathfrak{n}}
\global\long\def\z{\mathfrak{z}}

\global\long\def\ad{\mathrm{ad}}
\global\long\def\Ad{\mathrm{Ad}}
\global\long\def\Aut{\mathrm{Aut}}
\global\long\def\Ctn{\mathrm{Ctn}}
\global\long\def\End{\mathrm{End}}
\global\long\def\GL{\mathrm{GL}}
\global\long\def\id{\mathrm{id}}
\global\long\def\Int{\mathrm{Int}}
\global\long\def\Ker{\mathrm{Ker}}
\global\long\def\Lie{\mathrm{Lie}}
\global\long\def\Reg{\mathrm{Reg}}
\global\long\def\rk{\mathrm{rk}}

\global\long\def\tg#1{\overline{#1}}

\title{Cartan subgroups  and regular points of  o-minimal  groups}
\author{El\'ias Baro}
\address{Departamento de \'Algebra, Facultad de Matem\'aticas, Universidad Complutense de Madrid, 
28040 Madrid, Spain. 
E-mail address: ebaro@ucm.es}
\author{Alessandro Berarducci}
\address{Universit\`a di Pisa, Dipartimento di Matematica, Largo Bruno Pontecorvo 5, 56127 Pisa, Italy. 
E-mail address: alessandro.berarducci@unipi.it}
\author{Margarita Otero}
\address{Departamento de Matem\'aticas, Universidad Aut\'onoma de Madrid, 28049 Madrid, Spain. 
E-mail address: margarita.otero@uam.es}

\thanks{The first and third authors are partially supported by  Spanish MTM2014-55565-P and Grupos UCM 910444. The second author was partially supported by PRIN 2012 Logica, Modelli e Insiemi and by Progetto di
Ricerca d'Ateneo 2015. Connessioni fra dinamica olomorfa, teoria ergodica e logica matematica nei sistemi dinamici.}


\date{5 July 2017}

\begin{document}
\begin{abstract} Let $G$ be a group definable in an o-minimal structure $\mathcal{M}$. We prove that the union of the Cartan subgroups of  $G$ is a dense subset of  $G$.  When $\mathcal{M}$ is an expansion of a real closed field we give a characterization of Cartan subgroups of $G$ via their Lie algebras which allow us to prove  firstly, that every Cartan subalgebra of the Lie algebra of $G$ is the Lie algebra of a definable subgroup -- a Cartan subgroup of $G$ --, and secondly, that the set of regular points of $G$ -- a dense subset of $G$ -- is formed by points which belong to a unique Cartan subgroup of $G$.
\end{abstract}
\maketitle

\section{Introduction}
Let $G$ be an arbitrary group, a subgroup $H$ of $G$ is called
a \emph{Cartan subgroup} (in the sense of Chevalley) of $G$
if $H$ is a maximal nilpotent subgroup of $G$, and for any subgroup
$X\unlhd H$ of finite index in $H$, $X$ has finite index in its
normalizer $N_{G}(X)$.

Cartan subgroups have been studied mainly in two classes of groups: algebraic groups and Lie groups. The study in the first class goes via the rigidity of the algebraic context and in the second one through the nice behaviour of analytic functions and the possibility of consider integration, that is, through the exponential map. For example, Bourbaki works  with Lie groups over non-discrete complete ultrametric fields where it makes sense to talk of analyticity and  the exponential map.

Here we work with \emph{groups definable in o-minimal structures}, where we  neither have such a rigidity  as in the algebraic  case nor the nice behaviour of analytic functions. Even worse, integration -- in general --  makes no sense. 
However, at the kernel of the theory there is a finiteness phenomena that allows us to define dimension and  has a good relation with the algebraic context (in the linear case). More important, even though we cannot  integrate,  when our o-minimal structure is  an expansion of a real closed field, we can take derivatives, so Lie algebras make sense. With these tools we will be able to develop a theory concerning both Cartan subgroups and regular points in  groups definable in  o-minimal expansions of real closed fields. {As a general remark, it is worth mentioning that} the regular points in Lie groups have nice properties from which properties of the Cartan subgroups can be deduced. In definable groups we work the other way around: firstly, we  prove  properties of Cartan subgroups making use of ideas of K.H. Neeb  and from there we get new insight about regular points. {Some of the techniques used could be useful in other weak contexts where there is a lack of analiticity  or algebraic structure.}

{In \S 3 we answer positively the main questions left open in \cite{BJO14Cartan}. The study  of Cartan subgroups  of groups definable in  o-minimal structures  was initiated in \cite{BJO14Cartan}, where it was proven that Cartan subgroups of a definable group exist, are definable, and fall into finite many conjugacy classes \cite[Thm.\,1]{BJO14Cartan}. In  Corollary \ref{answers} we prove, among other results, that the union $X$ of the Cartan subgroups of a given definable group  $G$ is a dense subset of  $G$ (see \cite[Question\,84]{BJO14Cartan}).} The latter implies that the set $X$ is \emph{syndetic} in $G$ (also called ``generic'')  i.e., there is a finite $E\sub G$ such that $XE=G$, and we conjecture (Conjeture\,\ref{conj1}) that it suffices to consider the conjugates of a specially chosen Cartan subgroup to get a syndetic subset of $G$.  As far as we know, this is not known even when $G$ is  a simple Lie group.
 
 In \S 4, working with  groups definable in  o-minimal expansions of real closed fields, we give a characterisation of  the concept of Cartan subgroup of a definable group inspired by the corresponding characterization for Lie groups given by  H.K. Neeb in \cite{Neeb96Weaklyexp}. In general, the possible lack of the exponential map (when working with o-minimal structures) implies the lack of an analogue to the Lie correspondence. Nevertheless, some examples of realisations  of Lie subalgebras as Lie algebras of subgroups of a definable group are given in \cite[Claim 2.31]{PPS00Def} via normalizers and centralizers (see Remark\,\ref{norcen}).  Here, the mentioned characterisation allows us to prove that  Cartan subalgebras can be realised as Lie algebras of Cartan subgroups (Corollary\,\ref{Liecorr}).
 
{Finally, in \S 5, we introduce the concept of regular element of a definable group $G$.
Given $g\in G$, the \emph{ rank} $r(g)$ of $g$ is the dimension of the maximal subspace $\g^{1}(\Ad(g))$ of the eigenvalue $1$
of the endomorphism $\Ad(g):\g\to\g$,
and one says that $g$ is regular if the rank function $r$ 
is locally constant around $g$. In the context of Lie groups, it follows immediately from analyticity that the rank assumes a global minimum around any regular point. Moreover, using the exponential map it is possible to show that the global minimum equals the rank
of the Lie algebra $\g$ and that the regular points belong to a unique Cartan subgroup of $G$. Even though we are not working in an analytic context, we will be able to prove all these properties of regular points  in Proposition\,\ref{prop:reg} and Corollary\,\ref{cor:reguniqueCar}. We would like to stress that the path we use could be interesting also in the Lie context. The classical approach \cite{BourbLALG7-9} uses the exponential map to relate the regular points of the group and those of its Lie algebra. We will instead quotient by the centre to assume that $G$ is a definable subgroup of  $\GL(n,R)$, and then we will be able to relate the regular points of $G$ and those of its Zariski closure (Proposition \ref{prop:regequivalences}).}

\section{Preliminary notions}\label{laycsa}

\emph{Definable groups.} We recall that an o-minimal structure $\mathcal{R}=(R,<,\ldots)$
is an ordered set $(R,<)$ equipped with some relations or functions
(for instance field operations) with the property that every subset
of $R$ definable in $\mathcal{R}$ (in the sense of first order logic)
is a finite union of points and intervals $(a,b)$ with $a,b\in R\cup\{\pm\infty\}$.
By the Tarski-Seidenberg result, an ordered field is o-minimal if
and only if it is real closed, so in particular the field of real
numbers is o-minimal. Another example of o-minimal structure -- given by Wilkie \cite{Wilkie96exp}-- is the
ordered field of real numbers with the real exponential function. 

The definition of o-minimality is a condition on definable subsets
of $R$, but it has nice consequences for definable subsets of $R^{n}$.
By the \emph{cell decomposition theorem}, any definable subset of $R^{n}$
can be decomposed into finitely many \emph{cells} (see Proposition \ref{reglarge}) and every definable
function $f:R^{n}\to R$ is piecewise continous, namely it is continous
on each cell of a decomposition of its domain. Here the continuity
refers to the topology induced by the order of $R$. The cell decomposition theorem 
provides also a natural geometric notion of dimension on definable sets.

We shall often assume that $\mathcal{R}$ expands a field (necessarily
real closed). In this case every definable set can be definably triangulated
and every definable function is piecewise differentiable. By the Tarski-Seidenberg
results, when $\mathcal{R}$ is a real closed field with no additional
structure, the definable sets are just the semialgebraic ones. More
generally, the definable sets in an o-minimal structure share many
of the tameness properties of semialgebraic and subanalytic sets,
as the above results already suggest. 

Let us now introduce the notion of definable group. A group $G$ is
definable in the o-minimal structure $\mathcal{R}$ if its domain
is a definable subset of $R^{n}$, for some $n$, and the graph of the
group operation is a definable set. Such groups have been studied
since the 80's, two pioneering papers being \cite{Pillay88} and \cite{Strzebonski94}. 
When $\mathcal{R}$ has field operations, the general
linear group $GL(n,R)$ and its (semi)algebraic subgroups are definable.
By \emph{definable choice} (which holds when $\mathcal{R}$ expands
a field, or even a group), given a definable equivalence relation
$E$ on a definable set $X$, there is a definable choice of representatives
for the equivalence classes, so we can identify $X/E$ with a definable
set. In particular, given a definable group $G$ and a definable normal
subgroup $H$, the quotient group $G/H$ is definable (this holds
in any o-minimal structure \cite{Edmundo03}). 

We can define the algebraic closure $K=R[i]$ of $R$ via the identification
$R[i]=R^{2}$, and speak about algebraic groups over algebraically
closed fields. Thus an elliptic curve over $R$ or $K$ is definable,
and so is any abelian variety (we need to observe that the projective
space $\mathbb{P}^{n}(K)=K^{n+1}/\sim$ can be identified with a definable
set). More examples of definable groups can be found in \cite{Strzebonski94}. Thanks to the above observations, \emph{all the results of this paper apply in particular  to algebraic groups over a real closed or algebraically closed field}. 

Notice that in the definition of definable group, we do not assume
that the group operation is continuous (with respect to the topology
induced by the ambient space $R^{n}$), however \cite{Pillay88} shows
that any definable group $G$ in $\mathcal{R}$ has a unique topology,
called \emph{t-topology}, which makes $G$ into a topological group
and coincides with the topology induced by $R^{n}$ on a large definable
subset $V\subseteq G$, where \emph{large} means of codimension $<\dim(G)$, see Remark \ref{largedense}. The t-topology
makes $G$ into a \emph{definable manifold} (locally definable homeomorphic
to $R^{n}$, with a finite atlas) and when the underlying set of $\mathcal{R}$
is the real line we get an actual topological manifold, so every definable
group in an expansion of the reals is a real Lie group.  When $\mathcal{R}$ is not based on the reals, the order topology
of $R$ is neither locally compact nor locally connected, but there
are definable substitutes for these notions. One says that $G$ is
\emph{definably compact}  if every definable curve $f:(a,b)\to G$
($a,b\in R)$ has a limit in $G$ in the t-topology, and $G$ is \emph{definably connected}
if it has no definable subgroups of finite index. The latter is equivalent
to say that $G$  cannot be partitioned
into two non-empty definable open subsets, which in turn is  equivalent  to the condition
 that any two points of
$G$ can be joined by a definable continous curve $f:(a,b)\to G$.
The connected component $G^o$ of $G$ is the intersection of all
the subgroups of $G$ of finite index. It can be shown that $G^o$
is definable and has finite index in $G$.

\emph{When $\mathcal{R}$ is an expansion of an arbitrary real closed field the definable group manifold can be definably embedded in some  $R^{m}$ so that the t-topology coincides with that of the ambient space}. In this case, definably compact is equivalent to closed and bounded.
Definably compact groups are well understood and there are strong
functorial connections with real Lie groups via the quotient 
by the ``infinitesimal subgroup'' \cite{EO04DCAG, BOPP05DCC,HPP08NIP}.
Much less is known in the non-compact case (but see \cite{Conversano13}), and this paper is a contribution in this direction.  See \cite{Otero08} for basic properties of definable groups.

Except in a few  stated cases, we work over an o-minimal structure $\mathcal{R}$ expanding a
real closed field $R$, $K$ will denote its algebraic closure and $F$ either $R$ or $K$. By \emph{definable}  we mean definable in  $\mathcal{R}$.

\

\emph{Lie algebras and Cartan subalgebras.} All the Lie algebras we are going to consider are  finite dimensional Lie algebras  over  $F$, unless otherwise stated. If 
$\g$ is a Lie algebra over $R$ we will  denote by $\g_{K}:=\g\otimes K$   its complexification, a Lie algebra over $K$. 

We recall some basic definitions and facts about Lie algebras (see \cite[Ch. VII]{BourbLALG7-9}).  
We will use the following notation. Let $V$ be a  finite dimensional $F$-vector space,  $\ff\in End(V)$ and $s\in F$ then $${V}^s(\ff):= \{X\in V\st (\ff-s\id)^n(X)=0\text{ for some } n\in\N\}.$$
 A subalgebra $\h$ of a Lie algebra $\g$ is a \emph{Cartan
subalgebra}  of $\g$ if it is nilpotent and $\h=\mathfrak{n}_{\g}(\h):=\left\{ X\in\g:\left[X,\h\right]\subseteq\h\right\} $. A Cartan subalgebra of $\g$  is maximal among nilpotent subalgebras of $\g$. If $\h$ is a Cartan subalgebra of $\g$ then its complexification $\h_{K}$ is a Cartan subalgebra of $\g_{K}$.   

 Let $\g$ be a Lie algebra and $\h$ a nilpotent subalgebra. Let 
$\rho:\h_K\to \gl({\g_K}): Z\mapsto \ad(Z)$ be the restriction of the adjoint representation of $\g_K$.
 A linear form $\lambda:\h_K \rightarrow K$ is a \emph{root} if there is $X\in \g_K$, $X\neq 0$, such that 
$\rho(Z)(X)=\lambda(Z)X, \text{ for all }Z\in \h \text{ (or equivalently for all  $Z\in\h_K$)}.$
 $\Lambda(\g,\h)$ denotes the set of roots of $\g$ with respect to $\h$. Given $\lambda\in \Lambda(\g,\h)$, 
$$\g_K^\lambda(\h_K):=\{X\in \g_K: \forall Z\in \h_K, [\rho(Z)-\lambda(Z)\id]^n(X)=0 \text{ for some } n \in \N\}.$$
{Note that if  $\lambda\in \Lambda(\g,\h)$ is such that $\lambda(\h)\sub R$ then} $\g_K^\lambda(\h_K)$ is the complexification of 
$$\g^\lambda(\h):=\{X\in \g: \forall Z\in \h, [\rho(Z)-\lambda(Z)\id]^n(X)=0 \text{ for some } n \in \N\}.$$

 By Lie's Theorem and since $\h$ is solvable, roots exist. Moreover, $\g_K=\bigoplus_{\lambda\in \Lambda} \g_K^\lambda(\h_K)$.
In fact, for each $Z\in \h_K$ we have that $\{\lambda(Z):\lambda\in \Lambda(\g,\h)\}$ is the set of eigenvalues of $\rho(Z)$ and $\g_K=\bigoplus_{\lambda\in \Lambda} \g_K^{\lambda(Z)}(\h_K)$.

Clearly $0: \h_K\rightarrow K:Z \mapsto 0$ is a root since  $\h$ is nilpotent. Moreover,  we have the following. 

i) $\h \subseteq \g^0(\h)$; 

{ii)} $\mathfrak{\mathfrak{n}}_{\g}(\g^{0}(\h))=\g^{0}(\h)$, in particular
the same holds for the complexifications, and

{iii) }$\h$ is a Cartan subalgebra of $\g$ of if and only if $\h=\g^{0}(\h)$.

Next we consider our  definable context. We will use results of \S 2.3 of \cite{PPS00Def}  and  \S 2 of \cite{PPS02Lin} concerning basic facts of Lie algebras over $R$.  

Given a definable $C^{1}$-manifold $M$ and a point $a\in M$ we
define the \emph{tangent space} $T_{a}(M)$ as the set of all equivalence
classes of definable $C^{1}$-curves $\alpha:I\to M$ with $\alpha(0)=a$,
where two curves are equivalent if they are tangent at $0$. We denote
by $\tg{\alpha}$ the equivalence class of $\alpha$ and we endow
$T_{a}(M)$ with the natural vector space structure as in the classical
case. 

\begin{rem}
\label{tg-space}Given a local chart $\varphi:U\to R^{n}$
around $a\in U\subseteq M$ we can identify $T_{a}(M)$ with $R^{n}$ via the isomorphism 
sending $\tg{\alpha}$ to $(\varphi\circ\alpha)'(a)$. 
\end{rem}
\begin{fact}
Given a definable finite dimensional vector space $V$ we identify $T_{e}(\GL(V))$
with $\End(V)$ via the natural isomorphism 
sending $\tg{\alpha}$ to $\alpha'(0):=\lim_{t\to0}\frac{\alpha(t)-\alpha(0)}{t}$
where the curve $\alpha$ takes values in $\GL(V)$ and the limit is
taken in $\End(V)\supseteq \GL(V)$. 
\end{fact}
\begin{rem}\label{sumprod}{If $\tg{\alpha},\tg{\beta}\in T_{e}(\GL(V))$ then 
$\tg{\alpha}+\tg{\beta}=\tg{\alpha\cdot\beta}$, where $(\al\cdot\be)(t)=\al(t)\circ\be(t)$.}
\end{rem}

\section{Cartan subgroups}

{The main purpose of this section is to prove that given a definable
group $G$ in an arbitrary o-minimal structure $\mathcal{M}$, all Cartan subgroups of $G$ have the same dimension and their union is dense in $G$ (Corollary \ref{answers}). We will show that both problems can be reduced to the case of $G$ being a subgroup of $\GL(n,R)$ definable in an o-minimal expansion $\mathcal{R}$ of a field. In the latter context, the two key ingredients will be the relation between Cartan subgroups of $G$ and those of its Zariski closure in $\GL(n,R)$ (Proposition \ref{chev}) and a kind of ''identity principle'' for definable groups (see Proposition \ref{Zcap} and Remark \ref{identprinc}).}

We survey some preliminary results proved in \cite{BJO14Cartan} for  o-minimal structures (not necessarily expansions of  real closed fields).

\begin{fact}\label{Cartan1} 
Let $G$ be a  group definable in an o-minimal structure $\mathcal{M}$. Then,

\emph{(1)\,\cite[Thm.\,1]{BJO14Cartan}} Cartan subgroups of  $G$ exist, are definable in $\mathcal{M}$ and they fall into finite many conjugacy classes, and
 
%


\emph{(2)\,\cite[Cor.\,75]{BJO14Cartan}} if $G$ is definably connected and $H$ is a Cartan subgroup
of $G$ then $H=C_{G}(H^o)H^o$, in particular if $H_{1}$ and
$H_{2}$ are Cartan subgroups of $G$ then $H_{1}^o=H_{2}^o$
implies $H_{1}=H_{2}$.
\end{fact}
Let $G$ be a  group definable in an  o-minimal structure $\mathcal{M}$.  We recall that a \emph{Carter subgroup} of $G$ is a definably connected
nilpotent subgroup of $G$ which has finite index in its normalizer. 
\begin{rem}
\label{Carter-Cartan}Let $G$ be a  group definable in an o-minimal structure $\mathcal{M}$. For each Carter
subgroup $Q$ of $G$, $H:=C_G(Q)Q$ is the unique  Cartan subgroup  of $G$ containing $Q$ by \cite[Lem.\,5]{BJO14Cartan}.
The definably connected component of a Cartan subgroup is a Carter subgroup.  Moreover, if $Q$ is a Carter subgroup of $G$, then $Q$ is contained in a maximal nilpotent 
subgroup $\widetilde{Q}$ of $G$, and any such subgroup $\widetilde{Q}$ is a Cartan subgroup of $G$ with $\widetilde{Q}^o=Q$ (see \cite[Lem.\,5\,(b)]{BJO14Cartan}).
\end{rem}

We go back to our setting of groups definable in our o-minimal expansion $\mathcal R$ of a real closed field $R$, where we can make sense of the Lie algebra of  a definable group.  \emph{For a definable group, its Lie algebra will be denoted (unless otherwise stated)  by the corresponding lower case letter in  gothic font.
}

\begin{rem}\label{norcen}
We recall that if $G$ is a definable group and $\h$ is a subspace
of its Lie algebra $\g$ then $\Lie(N_{G}(\h))=\n_{\g}(\h)$
and $$\Lie(Z_{G}(\h))=\mathfrak{z}_{\g}(\h):=\left\{ X\in\g:\left[X,\h\right]=0\right\}, $$
where 
\[
N_{G}(\h):=\left\{ g\in G: \Ad(g)(\h)\subseteq\h \right\} 
\text{ and }
Z_{G}(\h):=\left\{ g\in G:\Ad(g)_{|\h}=\id_{\h}\right\} 
\]
 are definable subgroups of $G$. Also that if $H$ is a normal definable
subgroup of $G$ and $\h$ its Lie algebra (an ideal of $\g$) then
$\Lie(G/H)=\g/\h$.
\end{rem}
\begin{lem}\label{normycen}
Let $G$ be a definable group and let $H$ be a definable connected subgroup
of $G.$ Then $N_{G}(H)=N_{G}(\h)$ and $Z_{G}(H)=Z_{G}(\h)$.
\end{lem}
\begin{proof}
If $g\in N_{G}(H)$ then $H^{g}\subseteq H$ so that $\Ad(g)\h\subseteq\h$.
For the converse, take $g\in N_{G}(\h)$ then $\Lie(H^{g})=\Ad(g)(\h)=\h=\Lie(H)$
and therefore $H^{g}=H$, since $H$ is definably connected.

On the other hand, if $g\in Z_{G}(H)$ then the map $\Int(g):H\to H:h\mapsto ghg^{-1}$
is the identity in $H$ so its derivative $\Ad(g)_{|\h}$ is the identity
on $\h$. For the other inclusion, take $g\in Z_{G}(\h)$ and observe
that since $H$ is definably connected,
$\Ad(g)_{|\h}=\id_{\h}$ implies $\Int(g)_{|H}=\id_{H}$.
\end{proof}

\begin{lem}
\label{LAcomm}Let $G$ be a definable group and let
$H_{1}$ and $H_{2}$ be definable subgroups such that the commutator $H:=[H_{1},H_{2}]$
is definable. Then, $[\h_{1},\h_{2}]\subseteq\h$.
\end{lem}
\begin{proof}
Let $\tg{\alpha}\in\h_{1}$ and $\tg{\beta}\in\h_{2}$.
Then, by definition 
\begin{align*}
[\tg{\alpha},\tg{\beta}] & =\ad(\tg{\alpha})(\tg{\beta}) =(\Ad\circ\alpha)'(0)(\tg{\beta}) =\lim_{t\to0}\frac{1}{t}\left((\Ad\circ\alpha)(t)-I\right)(\tg{\beta})\\
 & =\lim_{t\to0}\frac{1}{t}\left(\Ad(\alpha(t))(\tg{\beta})-\tg{\beta}\right) =\lim_{t\to0}\frac{1}{t}\left(\tg{\Int(\alpha(t))\circ\beta}-\tg{\beta}\right)\\
 & =\lim_{t\to0}\frac{1}{t}\left(\tg{s\mapsto[\alpha(t),\beta(s)]}\right)\in\h
\end{align*}
as desired (the last equality follows by Remark\,\ref{sumprod}). 
\end{proof}

Let $G$ be a definably connected  group. Then, clearly $G$ is abelian if and only $\g$ is abelian.  In \cite{PPS00Def} it is proved that  $G$ is semisimple if and only $\g$ is semisimple. We next prove the corresponding statement for nilpotency and solvability. We first need the following.
\begin{fact}[{\cite[\S\,6]{BJO12Comm}}]
\label{comm}  Let $G$ be a solvable definable group.
Then the subgroups of $G$ of the lower central series and the derived
series of $G$ are definable, and definably connected if so it is
$G.$
\end{fact}

\begin{prop}
\label{defgpLA}Let $G$ be a definable group
and $H\leq G$ definably connected. Then, 

$(1)$ $H$ nilpotent if and only if $\mathfrak{h}$ nilpotent;

$(2)$ $H$ solvable if and only if $\mathfrak{h}$ solvable, and 

$(3)$ $H$ Carter subgroup of $G$ if and only if $\mathfrak{h}$ Cartan
subalgebra of $\mathfrak{g}$.
\end{prop}
\begin{proof}
Let $H$ be nontrivial. $(1)$ Suppose first that $H$ is nilpotent.
Let $\mathcal{C}^{1}\mathfrak{h}$, $\mathcal{C}^{2}\mathfrak{h}$, $\dots$
be the lower central series of $\mathfrak{h}$. Let $\mathcal{C}^{1}H, \mathcal{C}^{2}H,\dots$
be the lower central series of $H$, which are all definable by Fact \ref{comm}
since $H$ is solvable. Then, by Lemma\,\ref{LAcomm}, $\mathcal{C}^{1}\mathfrak{h}:=[\mathfrak{h,\mathfrak{h}}]\subseteq \Lie([H,H])=\Lie(\mathcal{C}^{1}H)$
and by induction $\mathcal{C}^{m}\mathfrak{h}\subseteq \Lie(\mathcal{C}^{m}H)$.
Hence $\mathcal{C}^{m}\mathfrak{\mathfrak{h}}=0$ for some $m$. Now suppose
$\mathfrak{h}$ is nilpotent and we prove by induction on the dimension
of $H$ that $H$ is nilpotent. Suppose $\dim H>1$ and let $Z(H)$
be its centre. If $Z(H)$ is finite $\mathfrak{z}(\mathfrak{h})=\Lie(Z(H))$
is trivial, which is impossible since $\mathfrak{h}$ is nilpotent. Hence $\dim H/Z(H)<\dim H$,
and since $\mathfrak{h}/\mathfrak{z}(\mathfrak{h})=\Lie(H/Z(H))$ is nilpotent
so is $H/Z(H)$, thus $H$ is nilpotent.

(2) The left to right implication is similar to case $(1).$ Suppose
$\mathfrak{h}$ is solvable. $H$ cannot be semisimple since $\mathfrak{h}$
semisimple would imply $\mathfrak{h}$ (and hence $H$) trivial.  Thus, there is 
 a definably connected nontrivial abelian normal subgroup $N$ of $H.$
Hence reasoning as above we get $H/N$ solvable and so $H$ solvable,
as required.

(3) By (1) and the fact that $\mathfrak{n}_{\mathfrak{g}}(\mathfrak{h})/\mathfrak{h}=\Lie(N_{G}(H))/\h=\Lie\left(N_{G}(H)/H\right).$
\end{proof}
\begin{cor}\label{equaldim}
Let $G$ be a definable group. Then, all the Cartan subgroups of $G$
have the same dimension.
\end{cor}
\begin{proof}
Let $H_{1}$, $H_{2}$ Cartan subgroups of $G$. Then $H_{1}^o$,
$H_{2}^{o}$ are Carter subgroups of $G$ by Remark\,\ref{Carter-Cartan}, and so their Lie algebras
are Cartan subalgebras of the Lie algebra of $G$ and therefore (since
the base field is of zero characteristic) of the same dimension
.
\end{proof}

The last corollary  answers positively  -- for o-minimal expansions of real closed fields -- a question left open in \cite{BJO14Cartan}. To  answer this and other questions of  \cite{BJO14Cartan} in their full generality  it will be convenient  to consider definable linear groups. The latter were first studied in \cite{PPS02Lin}.

Let $G$  be a  definable subgroup of $\GL(n,F)$. Let $\zar{G}{F}$ be  the Zariski closure of $G$ in $\GL(n,F)$. Then   $\zar{G}{F}$ is a definable group; it  is  the smallest algebraic subgroup of $\GL(n,F)$ that contains $G$. {Moreover, if $G\leq\GL(n,R)$ is definable then $\zar{G}{K}$ is defined over $R$ and $\zar{G}{K}\cap 
\GL(n,R)=\zar{G}{R}$.} A Lie subalgebra  of $\gl(n,F)$  is said to be \emph{algebraic}  if it is the Lie algebra of an algebraic subgroup of $\GL(n,F)$. Given  a Lie subalgebra $\g$ of $\gl(n,F)$, $a(\g)$ denotes the  minimal algebraic Lie subalgebra of $\gl(n,F)$ containing $\g$. {We recall that if $\mathfrak{a}$ is a subalgebra of $\gl(n,F)$ and $U$ and $V$
are linear subspaces of $\gl(n,F)$ such that $U\subseteq V$, then $\left[\mathfrak{a},V\right]\subseteq U$
implies $\left[a(\mathfrak{a}),V\right]\subseteq U$ (see \cite[Ch.VIII.3 p.\,112]{Hochschild81AlgGroups}). Also note that if $H$ is an algebraic subgroup of $\GL(n,R)$ then the Lie algebra of $\zar{H}{K}$ in $\GL(n,K)$ is $\h_K:=\h  \otimes K$.}

The first two conclusions  of the following proposition are proved in \cite[Ch.VI \S\,5 Lem.\,2]{Chevalley55AlgLie} for analytic linear groups.

\begin{prop}\label{zarcl} Let $G$ be  a definable subgroup of $\GL(n,F)$ and $\zar{G}{}$ its Zariski closure. Then,  $\Lie(\zar{G}{})=a(\Lie(G))$. Moreover, if $G$ is definably connected then  $\zar{G}{}$ is irreducible and $G$ is normal in $\zar{G}{}$. \end{prop}
 \proof
Suppose first that $G$ is definably connected. 
Let  $G_1$ be an algebraic irreducible subgroup of $\GL(n,F)$ such that $a(\g)=\Lie(G_1)$. Since $\g\sub a(\g)$ and $G$ is definably connected we have $G\sub G_1$ and so $\zar{G}{}\sub G_1$. On the other hand 
$\g\sub\Lie(\zar{G}{})$  and the latter being algebraic imply $a(\g)\sub\Lie(\zar{G}{})$, and since $G_1$ is irreducible we have $G_1\sub \zar{G}{}$. Therefore, $\zar{G}{}=G_1$ and  hence $\zar{G}{}$ irreducible and $\Lie(\zar{G}{})=a(\g)$, as required.  Next we prove that $G$ is normal in $\zar{G}{}$. Note first that $N_{\zar{G}{}}(\g)$ is algebraic and contains $G$, hence $\zar{G}{}=N_{\zar{G}{}}(\g)$. Since $G$ is definably connected, by Lemma\,\ref{normycen} we have $N_{\zar{G}{}}(\g)=N_{\zar{G}{}}(G)$.

For the general case, it suffices to note that $G$ -- being definable -- is a \emph{finite} union of translate of its  definably connected component $G^o$,  and so $\zar{G}{}$ is a finite union of translates of $\zar{G^o}{}$ and so $\Lie(\zar{G}{})=\Lie(\zar{G^o}{})$.
 \endproof
 
 As we have already mentioned above, all the Cartan subalgebras of a Lie algebra over a field of characteristic 0 have the same dimension, called the \emph{rank} of the Lie algebra and denoted by $\rk(-)$.
 
 \begin{cor}\label{rka(g)}  Let $G$ be  a definably connected subgroup of $\GL(n,F)$ and  $G_1$ its Zariski closure in  $\GL(n,F)$. Then,  $\rk\,\g_1=\rk\,\g+\dim\g_1-\dim\g$.
 \end{cor}
 \proof By  Proposition\,\ref{zarcl} $\g_1=a(\g)$.   So we have   that $\g$ is a subalgebra of $\gl(n,F)$, $F$ is a characteristic 0 field and $\g_1$ is the minimal algebraic Lie algebra containing $\g$. Hence, we are exactly under the hypothesis of  \cite[Ch.VI \S\,4 Prop.\,21]{Chevalley55AlgLie}, where  the required equality is concluded.
 \endproof
 
 Next, we adapt the proof of \cite[Ch.VI \S\,5 Prop.\,1]{Chevalley55AlgLie}  for analytic linear groups to our context which gives us the relationship between the Cartan subgroups of a definable linear group and those of its Zariski closure.
 
 \begin{prop}\label{chev} Let $G$ be  a definably connected  subgroup of $\GL(n,F)$ and $\zar{G}{}$ its Zariski closure. Then, the Cartan subgroups of $G$ are the intersection with $G$ of the Cartan subgroups of $\zar{G}{}$. 
 \end{prop}
 \proof   Let   $\g'=\Lie(\zar{G}{})$. Firstly, note that by  Proposition\,\ref{zarcl} $\zar{G}{}$  is an irreducible algebraic subgroup of $\GL(n,F)$  and so we are under the hypothesis of $n^o$\,2 of  \cite[Ch.VI \S\,4]{Chevalley55AlgLie}.
 
 Let $H'$ be a Cartan subgroup of $\zar{G}{}$. By Proposition\,\ref{defgpLA} $\h'$ is a Cartan subalgebra of $\g'$.  We first check that $\h'=a(\h'\cap\g)$ and that $\h'\cap\g$ is a Cartan subalgebra of $\g$.   By [\emph{Ibid.}\,Prop.5] $H'$ is  algebraic  and by  [\emph{Ibid.}\,Thm.2]  $H'$ is irreducible. Hence  $\h'\supseteq a(\h'\cap\g)$. For the other inclusion,  we can apply [\emph{Ibid.}\,Prop.21] and  get  first that $\h'\cap\g$ is a Cartan subalgebra of $\g$ since $\g'=a(\g)$, and hence  also  that $a(\h'\cap\g)$ is a Cartan subalgebra of $\g'$,  finally by maximality of the Cartan subalgebras we get the required equality. We must prove that $H'\cap G$ is a Cartan subgroup of $G$.  Since $\Lie(H'\cap G)=\h'\cap\g$  is a Cartan subalgebra  of $\g$, we have that $(H'\cap G)^o$ is a Carter subgroup of $G$ by Proposition\,\ref{defgpLA}, so by Remark\,\ref{Carter-Cartan} it suffices to prove that $H'\cap G$ is maximal nilpotent. Nilpotency is clear being a subgroup of the Cartan $H'$. 
 Let $H'\cap G\leq Q\leq G$ with $Q$ nilpotent. Let $Q':=\zar{Q}{}$, then $Q'$ is also nilpotent  and $\zar{(H'\cap G)^o}{}\sub Q'$. Note that $H'\sub Q'$. Indeed, since both subgroups are algebraic and $H'$ is irreducible, it suffices to prove that $\h'\sub\mathfrak q'$, and we have that $\h'=a(\h'\cap\g)=a(\Lie(H'\cap G))=a(\Lie((H'\cap G)^o))= \Lie(\zar{(H'\cap G)^o}{})\sub \Lie(Q')=\mathfrak q'$. Now, $H'$ being maximal nilpotent subgroup of $\zar{G}{}$ implies $H'=Q'$, and so $Q\sub Q'\cap G=H'\cap G$. Therefore, $H'\cap G$ is a Cartan subgroup of $G$.
 
 Let $H$ be a Cartan subgroup of $G$. By Proposition\,\ref{defgpLA} $\h$ is a Cartan subalgebra of $\g$. Then, the rest of the proof follows as in  \cite[Ch.VI \S\,5 Prop.1]{Chevalley55AlgLie}, here there are the details. Let $\h'=a(\h)$.  Again by  Prop.21 of \cite[Ch.VI \S\,4]{Chevalley55AlgLie}  $\h'$ is a Cartan subalgebra of $\g'$ and by  [\emph{Ibid.}\,Prop.5] $\h'=\Lie(H')$ for some Cartan subgroup $H'$ of $\zar{G}{}$. Then $H'\sub\zar{H}{}$. Indeed, since both are algebraic and $H'$ is irreducible, it suffices to check that $\h'\sub  \Lie(\zar{H}{})$. Now, $H\sub \zar{H}{}$ implies $\h\sub \Lie(\zar{H}{})$ and the latter is an algebraic Lie algebra so it also contains $a(\h)=\h'$. Then, since $H'$ is maximal nilpotent and $\zar{H}{}$ is nilpotent for $H$ is so, we have $H'=\zar{H}{}$. Hence $H\sub H'$,  and so  $H\sub H'\cap G$. Finally, we get  $H=H'\cap G$  since $H'\cap G$ is nilpotent and by maximal nilpotency of $H$.
 \endproof
 
 The next proposition  is a key fact to prove the density of the union of the Cartan subgroups in a definably connected linear group.
 
 \begin{prop}\label{Zcap} Let $G$ be a definably connected subgroup of $\GL(n,F)$.  Let $X\sub\GL(n,F)$ be  a Zariski closed set  such that $X\cap G$ has non-empty interior in $G$. Then, the Zariski closure $\zar{G}{}$ is contained in  $X$. 
 \end{prop}
 \proof
For $Z\sub \GL(n,F)$,  $\zar{Z}{}$ will denote  its Zariski closure. Let  $g\in \Int_G(X\cap G)$ and consider the open neighbourhood of the identity $V:=g^{-1}\Int_G(X\cap G)$. Now, for each open neighbourhood of the identity  $U\subseteq V$ consider its Zariski closure $\zar{U}{}$ in $\GL(n,F)$. 
 Let $$Y:=\bigcap_{U\sub V} \zar{U}{}.$$ Thus, $Y$ is a finite intersection and hence  $Y=\zar{U_0}{}$,  
 for some open neighbourhood of the identity  $U_0\subseteq V$, which we can assume to be symmetric, that is, $U_0^{-1}=U_0$. 

We claim that $Y$ is a subgroup of $\GL(n,F)$.  Clearly $Y\sub \zar{G}{}$. Let $U_1$ be an open symmetric neighbourhood of the identity in $G$ with $U_1U_1\subseteq U_0$. We first show that $YY=Y$. Take $h\in U_1$ and consider the algebraic subset $Z_h=\{x\in \zar{G}{}: hx\in Y \}$ of $Y$. If $x\in U_1$ then $hx\in U_1U_1\subseteq U_0\subseteq\zar{U_0}{}=Y$. That is, $U_1\subseteq Z_h$ and therefore $Y=\zar{U_0}{}\subseteq Z_h$.  Note that we have showed that the algebraic subset of $\zar{G}{}$ given by $\{h\in \zar{G}{}:hY\subseteq Y\}$ contains $U_1$,  therefore it contains its Zariski closure $\zar{U_1}{}=Y$. In particular, $YY=Y$. Now, let us see that $Y^{-1}=Y$. Take the algebraic set $W:=\{x\in \zar{G}{}: x^{-1}\in Y\}$. Since $U_1$ is symmetric we have that $U_1\subseteq W$ and therefore $\zar{U_1}{}=Y$ is also contained in $W$. Thus, $Y^{-1}\subseteq Y$, as required, what proves the claim.

Now, by the claim   $Y\cap G$ is a subgroup  of $G$, and  since it contains the open subset $U_0$, we can deduce that $\dim(Y\cap G)=\dim G$ and therefore, $Y\cap G=G$ by connectedness of $G$. That is, $G\leq Y$. In particular, $Y=\zar{G}{}$. Finally note that
$$\zar{G}{}\supseteq \zar{X\cap G}{}\supseteq\zar{gU_0}{}=g\zar{U_0}{}=g\zar{G}{}=\zar{G}{}$$
and hence  $\zar{G}{}=\zar{X\cap G}{}$,  therefore  $\zar{G}{}\sub X$.
\endproof


{
\begin{rem}\label{identprinc}The above is a rudimentary version for definable linear groups of the analytic `identity principle''. For, if $G$ denotes a connected closed subgroup  (or more generally, a connected analytic submanifold) of $\GL(n,\R)$ and $X$ denotes an analytic subset of $\GL(n,\R)$ such that $\Int_G(X\cap G)\neq \emptyset$, then clearly $G \sub X$. 
\end{rem}}

Let  $\Ctn(L)$ denote  the set of Cartan subgroups of a group $L$.

\begin{thm}\label{denselin} Let $G$ be a definably connected subgroup of $\GL(n,R)$. Then, the union of the Cartan subgroups of $G$ is dense in $G$.
\end{thm}
\proof Let $\zar{G}{}$ be the Zariski closure of $G$ in  $\GL(n,R)$. Then, as above,  we are under the hypothesis of  \cite[Ch.VI \S\,4  Def.\,2]{Chevalley55AlgLie}. Let $U$ be the set of regular elements of the algebraic group $\zar{G}{}$ (in the sense  of [\emph{Ibid.} Def.\,2], see below Proposition\,\ref{reg^a}). Then, by  [\emph{Ibid.} Prop.\,6] $U$ is a nonempty Zariski open subset of $\zar{G}{}$.  

\emph{Claim:}  The set $U\cap G$ is dense in $G$. 

Indeed, if not $(\zar{G}{}\setminus U)\cap G$ has nonempty interior in $G$, and by Proposition\,\ref{Zcap}  we  get $\zar{G}{}\setminus U=\zar{G}{}$, a contradiction.

On the other hand,   by  [\emph{Ibid.} Thm.2]  $U\sub \bigcup\{H'\colon H' \in \Ctn(\zar{G}{})\}$ and hence $$\bigcup\{H'\cap G\colon H' \in \Ctn(\zar{G}{})\}=\bigcup\{H'\colon H' \in \Ctn(\zar{G}{})\}\cap G$$ is dense in $G.$ Finally, by Proposition\,\ref{chev}  we have that $$ \bigcup\{H'\cap G\colon H' \in \Ctn(\zar{G}{})\}=\bigcup\{H\colon H\in \Ctn(G)\}.$$
\endproof

 We end this section by giving a positive answer to the main questions left open in \cite{BJO14Cartan}. 
\begin{rem}\label{largedense}Note that since groups definable in an o-minimal structure are definable manifolds, a definable subset is dense if and only if is large (small codimension). 
\end{rem}

\begin{cor}\label{answers} Let $G$ be a definably connected group definable in an  o-minimal structure $\mathcal{M}$.Then,

\emph{(1)} all the Cartan subgroups of $G$  have the same dimension;

\emph{(2)} the union of the Cartan subgroups of $G$ is dense in $G$;

\emph{(3)} Cartan subgroups of $G/R^o(G)$ are exactly of the form $HR^o(G)/R^o(G)$ with $H$ a Cartan subgroup of $G$, where $R(G)$ denotes the \emph{radical} of $G$, and 

\emph{(4)}\,for every Cartan subgroup $H$  of  $G$ and $a\in H$ such that, modulo $R^o(G)$, $a$ is in a unique
conjugate of $aH^o$, $(aH^o)^{R^o(G)}$ is dense in $aH^oR^o(G).$
\end{cor}
\proof  The statement of  \cite[Prop.\,85]{BJO14Cartan}  says that (2) implies (3) and (4). 

To prove (1) and (2) we first consider the canonical definable projection $G\to G/Z(G)$. By \cite[Lem.\,8]{BJO14Cartan} each Cartan subgroup of $G$ contains $Z(G)$ and  $$\Ctn(G/Z(G))=\{H/Z(G): H\in \Ctn(G)\}.$$Thus, it suffices to prove (1) and (2) for  $G/Z(G)$.  By Fr\'econ  Theorem \cite[Thm.\,5.15]{Frecon17Lin} $G/Z(G)=G_1\times\cdots\times G_m$, where, for each $i=1,\dots,m$,  $G_i$ is a definable group and there are a definable real closed field $R_i$ and an integer $n_i$  such that $G_i$ is a definable subgroup of $\GL(n_i, R_i)$. By  \cite[Cor.\,10]{BJO14Cartan} $$\Ctn(G/Z(G))=\{H_1\times\cdots\times H_m: H_i\in \Ctn(G_i), i=1,\dots,m\}.$$Thus,  clearly it suffices to prove (1) and (2) for  each $G_i$, and moreover we can assume that each $G_i$ is definable in an o-minimal expansion of $R_i$. We conclude by applying  Corollary\,\ref{equaldim} and Theorem\,\ref{denselin}.
\endproof

\begin{cor}\label{generic} Let $G$ be a definably connected group definable in an o-minimal structure $\mathcal{M}$. Then, the union of the Cartan subgroups of $G$,  $X:=\bigcup\Ctn(G)$, is syndetic in $G$, i.e., there is a finite $E\sub G$ such that $XE=G$.
\end{cor}
\proof  The set $X$ is dense in $G$ by Corollary\,\ref{answers}\,(2). Now, since $X$ is definable this is equivalent to $X$ being of small codimension, i.e. $\dim (G\setminus X)<\dim G$,  and this in turn implies that finitely many translates of $X$ covers $G$.
\endproof

In  \cite[Remark 57]{BJO14Cartan} we proved that  in the case of $\mathrm{SL}(2,\R)$, it suffices to consider the conjugates of a specially chosen Cartan subgroup (the diagonal matrices) to get a syndetic subset of $\mathrm{SL}(2,\R)$, and we conjecture that this is always the case.

\begin{conj}\label{conj1} Let $G$ be a definably simple group definable in an o-minimal structure $\mathcal{M}$. Then, there is a Cartan subgroup $H$ of $G$ such that  $H^G$ is syndetic in $G$.
\end{conj}

As we have already mentioned, this is not known even when $G$ is a simple Lie group.  Note also that if the definable group is not definably compact,  definably simple is equivalent to simple  \cite[Cor.6.3]{PPS02Lin}.

\begin{prop}If Conjecture \emph{\ref{conj1}} holds true, then every definable group $G$ has a Cartan subgroup $H$ of $G$ such that  $H^G$ is syndetic in $G$.
\end{prop}
\begin{proof}Let $R^o:=R^o(G)$ be the connected component of the radical of $G$. The quotient $G/R^o$ is a central product of finite-by-definably simple subgroups \cite[Remark 87]{BJO14Cartan}. Moreover, by \cite[Lemma 9]{BJO14Cartan}, a Cartan subgroup of $G/R^o$ is a product of Cartan subgroups of the factors. Therefore, it follows from the hypothesis that $G/R^o$ has a Cartan subgroup whose set of conjugates is syndetic. In other words, by Corollary \ref{answers}\,(3) there is a Cartan subgroup $H$ of $G$ such that $(HR^o)^G$ is syndetic in $G$.

On the other hand, $(H^o)^{R^o}$ is dense in $H^oR^o$. For, $H^o$ is a Carter subgroup of the solvable definably connected group $H^oR^o$, so it follows from  \cite[Thm 40]{BJO14Cartan}. Now, we claim that $H^{R^o}$ is dense in $HR^o$. Indeed, let $aH^o$ be a coset of $H^o$ in $H$.
By Corollary \ref{answers}.(3) we have that $HR^o/R^o$ is a Cartan subgroup of $G/R^o$. In particular, $H^oR^o/R^o$ is its connected component and $aH^oR^o/R^o$ a coset. By \cite[Thm 82]{BJO14Cartan} we can assume that, modulo $R^o$, $a$ is in a unique conjugate of $aH^o$. Therefore, by Corollary \ref{answers}.(4) we get that $(aH^o)^{R^o}$ is dense in $aH^oR^o$. In particular, we deduce that $H^{R^o}$ is dense in $HR^o$.

Finally, since $H^{R^o}$ is dense in $HR^o$, we get that $(H^{R^o})^G=H^G$ is dense in $(HR^o)^G$. Since $(HR^o)^G$ is syndetic in $G$ there is a finite $E\sub G$ such that $E(HR^o)^G=G $. Thus, $EH^G$ is dense in $G$, and so syndetic in $G$. It follows that $H^G$ is also syndetic in $G$, as required. 
\end{proof}

\section{A characterisation of Cartan subgroups}\label{sect:CharCar}

{In this section we fix a definably connected  group $G$ definable  as usual in our o-minimal expansion of a field $R$.
Given such $G$ with Lie algebra $\g$, and a Cartan subalgebra $\h$ of $\g$,
we want to show that there is a Cartan subgroup $H\leq G$ whose Lie algebra
is $\h$. The problem is that in the definable context we do not have
a general correspondence between analytic  subgroups and Lie subalgebras,
mainly because of the lack of an exponential map. To that aim, we shall introduce an alternative definition of Cartan
subgroups of $G$, given by Karl-Hermann Neeb in \cite{Neeb96Weaklyexp} for connected Lie}
groups (see Definition \ref{CCartan}).

 As in the classical case we have the following fact.
\begin{rem}
 Let $\h$  be a nilpotent
subalgebra of $\g$. Then there is a  definable
action of $N_{G}(\h)$ on the set of roots $\Lambda:=\Lambda(\g,\h)$
\[
\begin{array}{rcl}
N_{G}(\h)\times\Lambda & \rightarrow & \Lambda\\
(g,\lambda) & \mapsto & \lambda\circ \Ad(g)|_{\h_{K}}
\end{array}
\]
\end{rem}
\proof We use the notation introduced in \S\,\ref{laycsa}, in particular $\rho$ is the restriction  to $\h_K$ of the adjoint representation of $\g_K$. We  have to show that $\lambda\circ \Ad(g)|_{\h_{K}}\in\Lambda$, for each $g\in N_G(\h)$ and each $\lambda\in\Lambda$. By definition of $\Lambda$ there is $X\in\g_{K}$, $X\not=0$, such that $\rho(Z)(X)=\lambda(Z)X$
for all $Z\in\h$. On the other hand, recall that $\Ad(g^{-1})[X_{1},X_{2}]=[\Ad(g^{-1})X_{1},\Ad(g^{-1})X_{2}]$, for all {$X_1,X_2\in\g_K$}
so that, for all $Z\in\h$, 
\begin{equation*}
\begin{split}
\rho(\Ad(g^{-1})(Z))\Ad(g^{-1})(X)&=[\Ad(g^{-1})(Z),\Ad(g^{-1})(X)]=\Ad(g^{-1})[Z,X] \\ 
&=\Ad(g^{-1})\ad(Z)(X)=\Ad(g^{-1})\rho(Z)(X)\\
&=\Ad(g^{-1})\lambda(Z)(X)=\lambda(Z)\Ad(g^{-1})X\\
&=(\lambda\circ \Ad(g))(\Ad(g^{-1})(Z))\Ad(g^{-1})X.
\end{split}
\end{equation*}
This shows that $\lambda\circ \Ad(g)|_{\h_{K}}$ is a root. 
\endproof

\begin{definition}\label{CCartan}Let $\h$ be a Cartan subalgebra of $\g$. Then
\[
C(\h):=\{g\in N_{G}(\h):\lambda\circ \Ad(g)|_{\h_{K}}=\lambda,\text{ for all }\lambda\in\Lambda(\g,\h)\}.
\]
is a definable subgroup of $G$ which we call a \emph{$C$-Cartan
subgroup} of $G$. 
\end{definition}

Our aim is to prove that the concepts of Cartan  subgroup and $C$-Cartan subgroup coincide. We begin  by establishing some basic properties of $C$-Cartan subgroups and their Lie algebras.
\begin{lem}
\label{Lie-of-C-h}Let $\h$ be a Cartan subalgebra of $\g$.
Then $\Lie(C(\h))=\h$. 
\end{lem}
\begin{proof}
$C(\h)$ is the kernel of the action of $N_{G}(\h)$ on the finite
set $\Lambda$, hence it has finite index in $N_{G}(\h)$, and therefore
it has the same Lie algebra, namely $\mathfrak{n}{}_{\mathfrak{g}}(\h)$,
which is $\h$ since $\h$ is a Cartan subalgebra of $\mathfrak{g}$. 
\end{proof}
\begin{cor}
\label{C-Cartan}Let $H$ be a definable subgroup  of $G$. Then, $H$ is a $C$-Cartan subgroup of $G$
if and only if $\h$ is a Cartan subalgebra of $\g$ and $C(\h)=H$. 
\end{cor}
\begin{prop}
\label{NCar1}Let $H$ be a Cartan subgroup of $G$. Then, $\mathfrak h$ is a Cartan subalgebra of $\g$, and 
 $H$ is contained in the $C$-Cartan $C(\mathfrak h)$ of $G$.
 \end{prop}
\begin{proof}
By Proposition\,\ref{defgpLA}  $\mathfrak h :=\Lie(H)=\Lie(H^o)$  is a Cartan subalgebra of $\g$, so by definition $C(\mathfrak h )$ is a $C$-Cartan of $G$.
Let us show that $H\leq C(\mathfrak h )$. { By Lemma\,\ref{Lie-of-C-h}
$\mathfrak h =\Lie(C(\mathfrak h ))$, and therefore  $H^o=C(\h)^o \sub C(\h)$. On the other hand, by Fact \ref{Cartan1}(2)
we have that $H=H^oC_{G}(H^o)$. Let $g\in C_{G}(H^o)$. Then $\Ad(g)_{|\mathfrak h _K}=\id_{{\mathfrak h}_K}$, thus $g\in C(\mathfrak h)$. Therefore  $H=H^oC_{G}(H^o)\sub C(\mathfrak h)$.}
\end{proof}
\begin{prop}
\label{NCar2}Let $\h$ be a Cartan subalgebra of $\g$. Then
$C(\h)^o$ is a Carter subgroup of $G$. Moreover, if $C(\h)$
is nilpotent then it is a Cartan subgroup of $G$. 
\end{prop}
\begin{proof}
By Lemma\,\ref{Lie-of-C-h} $\Lie(C(\h)^o)=\h$. The latter being a Cartan subalgebra of $\g$, we get that $C(\h)^o$ is a Carter subgroup of $G$ by Proposition\,\ref{defgpLA}. Then, by Remark\,\ref{Carter-Cartan} there is a unique Cartan subgroup $H$ of $G$ containing $C(\h)^o$, and $H^o=C(\h)^o$.
 By  Proposition\,\ref{NCar1}
we have that $H\leq C(\h)$, so if $C(\h)$ is nilpotent, then by maximality
of $H$ we get that $H=C(\h)$. 
\end{proof}
By the above, to show that the concepts of Cartan and $C$-Cartan subgroups of $G$ coincide, it remains to prove that the $C$-Cartan subgroup are nilpotent. We do that  by {reducing the problem to the linear case (see Remark \ref{outlineNeeb} below).}

Consider $\Ad:G\rightarrow\Aut(\g_{K})$. The Lie algebra of $\Aut(\g_{K})\subseteq \GL(\g_{K})$
is $Der(\g_{K})$, which is a subset of $\mathfrak{gl}(\g_{K})$.
Let $G_{1}:=\zar{\Ad(G)}{K}$. Hence $\g_{1}\subseteq Der(\g_{K})\subseteq\mathfrak{gl}(\g_{K})$. 

Given a subalgebra $\mathfrak{b}$ of $\mathfrak{gl}(\g_{K})$, let us consider   
$\mathfrak{b}_{s}:=\left\{ X_{s}:X\in\mathfrak{b}\right\} $ and 
$\mathfrak{b}_{n}:=\left\{ X_{n}:X\in\mathfrak{b}\right\} $, where $X_s$ and $X_n$ denote the semisimple and nilpotent parts of the endomorphism $X$ in the additive Jordan decomposition. We recall that if $\mathfrak{b}$ is an algebraic Lie algebra then
$\mathfrak{b}_{s}$, $\mathfrak{b}_{n}\subseteq\mathfrak{b}$ (see
\cite[Ch.V Thm.2.3]{Hochschild81AlgGroups}). 

\begin{remark}\label{outlineNeeb}{In the context of Lie groups, H.K. Neeb shows that $C(\h)$ is nilpotent by proving that $\Ad(C(\h))$ is contained in the Zariski closure $\zar{\text{exp}(\ad(\h)) }{}\sub \Aut(\g_\mathbb{C})$,  which is a Cartan subgroup of $G_1$. For, A. Borel \cite[12.6]{Borel91LinearAlgGroups} proves that Cartan subgroups of $G_1$ are exactly the centralizers of maximal tori. On the other hand, $\zar{\text{exp}(\ad(\h)) }{}$ turns out to be $Z_{G_1}(T)$ where $T=\zar{\text{exp}(\{(\ad X)_s:X\in \h\})}{}$ is a maximal complex torus of $G_1$.}

{We will avoid the use of the exponential map in the following way. Denote $\n:=\ad(\h_{K})$. We first show that the Zariski closure of $\Ad(C(\h))$ inside $G_{1}$ coincides with the centralizer $Z_{G_{1}}(\n_{s})$. In view of the argument in the paragraph above, $Z_{G_{1}}(\n_{s})$ should be a Cartan subgroup of $G_{1}$. Indeed, using results of abelian algebraic groups, we show that it is irreducible, nilpotent and equals the irreducible component of its normalizer. This is enough since we work in an algebraically closed field.}
\end{remark}

{We will use the following characterization of C-Cartan groups whose proof can be carried out word by word in our context.}

\begin{fact}[{\cite[Lem. A.2]{Neeb96Weaklyexp}}]
\label{NCartan}Let $\h$ be a Cartan subalgebra of $\g$. Then
\[
C(\h)=\{g\in N_{G}(\h):\Ad(g)\ad(X)_{s}=\ad(X)_{s}\Ad(g)\text{ for all }X\in\h\}.
\]
(or alternatively for all $X\in\h_{K}$). 
\end{fact}

\begin{prop}
\label{centraliserCSA}Let $\h$ be a Cartan subalgebra of $\g$.
Then $\mathfrak{z}_{\g_{1}}(\mathfrak{n}_{s})$ is a Cartan subalgebra
of $\g_{1}$, where $\mathfrak{n}:=$ \textup{$\ad(\h_{K})$. Moreover
$\z_{\g_{1}}(\n_{s})=a(\n)$.}
\end{prop}
\begin{proof}
We first note that $\ad(\g_{K})\subseteq\g_{1}$. Indeed, take $X\in\g$
and consider 
\[
\ad(X):\g_{K}\rightarrow\g_{K}:Y\mapsto[X,Y]\in Der(\g_{K}).
\]
We have that $\Ad:G\rightarrow \Ad(G)\subseteq G_{1}\subseteq\Aut(\g_{K})$
so that the derivative is $\ad:\g\rightarrow\g_{1}\subseteq Der(\g_{K})$,
complexifing we get $\ad:\g_{K}\rightarrow\g_{1}\subseteq Der(\g_{K}).$

Now, since $\h$ is a Cartan subalgebra of $\g$, $\n$
is a Cartan subalgebra of $\ad(\g_{K})$, in particular $\n\subseteq\mathfrak{gl}(\g_{K})$
is nilpotent. By \cite[I.8(iv)]{Neeb94Closedness} we have that the
subalgebra $\mathfrak{n}_{s}$ of $\g_{1}\subseteq\mathfrak{gl}(\g_{K})$
is abelian, commutes with $\n$, and 
\[
a(\mathfrak{n})=\mathfrak{n}+a(\mathfrak{n}_{s}).
\]

Now, \cite[I.9]{Neeb94Closedness} says  that $a(\mathfrak{n})$
is a Cartan subalgebra of $a(\ad(\g_{K}))=\g_{1}$, thus $a(\n)=\z_{\g_{1}}(a(\n)_{s})$
by \cite[Ch.VII \S\,5 Prop.6]{BourbLALG7-9}.

On the other hand, since $\n$ commutes with $\n_{s}$ it also commutes
with $a(\n_{s})$. Hence 
$a(\mathfrak{n})_{s}=(\mathfrak{n}+a(\mathfrak{n}_{s}))_{s}=\n_{s}+a(\n_{s})_{s}=a(\mathfrak{n}_{s})_{s}.$
Thus, 
$$a(\n)=\z_{\g_{1}}(a(\n)_{s})=\z_{\g_{1}}(a(\n_{s})_{s})\supseteq\z_{\g_{1}}(a(\n_{s}))=\z_{\g_{1}}(\n_{s}).$$

Finally, we have that $[\n,\n_{s}]=0$ and so $a(\n)\subseteq\z_{\g_{1}}(\n_{s})$.
Therefore, we have that $\z_{\g_{1}}(\n_{s})=a(\n)$ is a Cartan subalgebra
of $\g_{1}$.
\end{proof}
\begin{cor}
\label{centralizerCartansubgroup}Let $\h$ be a Cartan subalgebra
of $\g$. Then $Z_{G_{1}}(\n_{s})$ is a Cartan subgroup of $G_{1}$,
where $\mathfrak{n}$ is \textup{$\ad(\h_{K})$.}
\end{cor}
\begin{proof}{First recall that for algebraic groups being irreducible is the same as being (Zariski) connected, and that $G_1$ is a irreducible algebraic subgroup of $\GL(n,K)$. Thus, by \cite[12.6]{Borel91LinearAlgGroups} it is enough to show that $Z_{G_{1}}(\n_{s})$ is irreducible, nilpotent and equal to the irreducible component of its normalizer.}

First we check that\emph{ $Z_{G_{1}}(\mathfrak{n}_{s})=Z_{G_{1}}(a(\mathfrak{n})_{s})$.}\,Clearly, $Z_{G_{1}}(a(\mathfrak{n})_{s})\subseteq Z_{G_{1}}(\mathfrak{n}_{s})$
because $\n_{s}\subseteq a(\n)_{s}$. For each $f\in G_{1}$, consider
the \emph{algebraic} subgroup $Z_{G_{1}}(f)=\{X\in G_{1}:fX=Xf\}$
of   $G_{1}$, then 
$
\Lie(Z_{G_{1}}(f))=\{X\in\g_{1}:fX=Xf\}.
$
Assume $f\in Z_{G_{1}}(\n_{s})$. Then we have that $\n_{s}\subseteq \Lie(Z_{G_{1}}(f))$
and therefore $a(\n_{s})\subseteq \Lie(Z_{G_{1}}(f))$. In particular
$a(\n)_{s}=a(\n_{s})_{s}\subseteq \Lie(Z_{G_{1}}(f))$ and therefore
$f\in Z_{G_{1}}(a(\n)_{s})$.


Let us prove that $Z_{G_{1}}(a(\mathfrak{n})_{s})$ is {irreducible}.
First note that since $a(\n_{s}$) is abelian and algebraic $a(\n_{s})=\Lie(A)$
where $A\leq \GL(\g_{K})$ is an abelian algebraic group. By \cite[4.7]{Borel91LinearAlgGroups}
$A=A_{u}\times A_{s}$, with the unipotent part $A_{u}$ and the semisimple part $A_{s}$   algebraic (and
abelian). Hence $a(\n_{s})=\Lie(A)=\Lie(A_{s})+\Lie(A_{u})$. Then $(a(\n)_{s}=)\,a(\n_{s})_{s}$
is $\Lie(A_{s})$. Since $A_{s}$ is abelian and consists of semisimple
elements (of $\GL(\g_{K})$) by \cite[8.4]{Borel91LinearAlgGroups}
is a subtorus of $G_{1}$. By \cite[11.12]{Borel91LinearAlgGroups}
$Z_{G_{1}}(A_{s})$ is {irreducible}. Now $Z_{G_{1}}(A_{s})=Z_{G_{1}}(\Lie(A_{s}))$
and $\Lie(A_{s})$ is $a(\n)_{s}$. Therefore, $Z_{G_{1}}(a(\mathfrak{n})_{s})$
is {irreducible}.%

{Finally, note that $\Lie(Z_{G_{1}}(\mathfrak{n}_{s})^o)=\z_{\g_{1}}(\n_{s})$ and therefore by Propositions\,\ref{centraliserCSA} and \ref{defgpLA} $Z_{G_{1}}(\mathfrak{n}_{s})^o$ is a Carter subgroup of $G_1$. It follows that $Z_{G_{1}}(\mathfrak{n}_{s})^o$ is nilpotent and $N_{G_1}(Z_{G_{1}}(\mathfrak{n}_{s}))^o=Z_{G_{1}}(\mathfrak{n}_{s})^o$. On the other hand, by Proposition \ref{zarcl} the nilpotent subgroup $\zar{Z_{G_{1}}(\mathfrak{n}_{s})^o}{K}$ is clearly the irreducible component of $Z_{G_{1}}(\mathfrak{n}_{s})$, and so $\zar{Z_{G_{1}}(\mathfrak{n}_{s})^o}{K}=Z_{G_{1}}(\mathfrak{n}_{s})$. Moreover,   
$$Z_{G_{1}}(\mathfrak{n}_{s})=\zar{Z_{G_{1}}(\mathfrak{n}_{s})^o}{K}=\zar{N_{G_1}(Z_{G_{1}}(\mathfrak{n}_{s}))^o}{K}$$ is the irreducible component of $N_{G_1}(Z_{G_{1}}(\mathfrak{n}_{s}))$, as required}.

\end{proof}
\begin{thm}
\label{Cartan=C-Cartan}The concepts of Cartan subgroup
and $C$-Cartan subgroup of a definably connected group $G$ coincide. 
\end{thm}
\begin{proof}
We first prove that every $C$-Cartan subgroup of $G$ is a Cartan
subgroup of $G$. By definition a $C$-Cartan is of the form $C(\h)$
for some Cartan subalgebra $\h$ of $\g$. By Corollary\,\ref{centralizerCartansubgroup}
$Z_{G_{1}}((\ad(\h_{K}))_{s})$ is a nilpotent subgroup of $G_{1}$($=\zar{\Ad(G)}{K}\sub Aut(\g_K)$). Now 
\begin{equation*}
\begin{split}
Z_{G_{1}}((\ad(\h_K)_s) &:=\{f\in G_{1}:\Ad(f)(\ad(X)_{s})={\ad(X)_{s}},\text{ for all }X\in\h_{k}\}\\
&\;= \{f\in G_{1}:f\circ \ad(X)_{s}=\ad(X)_{s}\circ f,\text{ for all }X\in\h_{k}\}
\end{split}
\end{equation*}
since in the linear case $\Ad$ is the usual conjugation. Hence $\Ad(C(\h))=\{\Ad(g):g\in C(\h)\}\subseteq Z_{G_{1}}((\ad(\h_{K}))_{s})$
by Fact \ref{NCartan}, and so $\Ad(C(\h))$ is nilpotent and thus so
is $C(\h)$ since the kernel of $\Ad$ is the centre of $G$. By Proposition\,\ref{NCar2}
$C(\h)$ is a Cartan subgroup of $G$.

Conversely, by Proposition\,\ref{NCar1} if $H$ is a Cartan subgroup of $G$,
then $H$ is contained in $C(\h)$ where $\h$ is the Lie algebra
of $H^o$. On the other hand $C(\h)$ is a Cartan subgroup, so it
must coincide with $H$. 
\end{proof}
 As we have mentioned, in our definable context we do not have  an analogue to the Lie correspondence. However, the last theorem  implies that  Cartan subalgebras can be realised as Lie algebras of Cartan subgroups,   as follows.

\begin{cor}\label{Liecorr}
Let $G$ be a definably connected group and $\h$ be a Cartan
subalgebra of $\g$. Then, there is a unique Cartan subgroup $H$ of
$G$ with Lie algebra $\h$. 
\end{cor}
\begin{proof}
Take $H:=C(\h)$ and use Lemma\,\ref{Lie-of-C-h}. For the  unicity clause, note that if $H_1$ and $H_2$ are Cartan subgroups of $G$ and $\Lie(H_1)=\Lie(H_2)$ then the corresponding Carters $H_1^o$ and $H_2^o$ coincide, thus  $H_1=H_2$ by Remark\,\ref{Carter-Cartan}.
\end{proof}
{We finish this section with a result from \cite{Hofmann92NearCartan} which will be useful in the study of the regular points.}
\begin{prop}\label{prop:reginCartan}Let $g\in G$ be  such that $$\g^1(\Ad(g)):=\{X\in\g\st(\Ad(g)-\id)^n(X)=0,\text{ for some }n\in\N\}$$is a Cartan subalgebra of $\g$. Then $g\in C(\g^1(\Ad(g)))$.
\end{prop}
\begin{proof}Let $\h:=\g^1(\Ad(g))$ and let $\Lambda:=\Lambda(\g,\h)$ be the set of roots. Let $E:=Span_K\Lambda \subseteq \h^*_K$.The action of  $\Ad(g)$ on $\Lambda$ induces the obvious automorphism
$$\widetilde{\Ad(g)}:E\rightarrow E \in Aut(E)\subseteq End(E).$$
Consider $\Gamma^*$ the group generated by $\widetilde{\Ad(g)}$ in $Aut(E)$. The group $\Gamma^*$ is finite because $\Ad(g)$ is a permutation on the finite set $\Lambda$. In particular, there is $m\in \mathbb{N}$ such that $[\widetilde{\Ad(g)}]^m=\id$, so that $t^m-1$ is null on the endomorphism $\widetilde{\Ad(g)}$. Since the roots of $t^m-1$ are simple, the endomorphism  is diagonalizable.

We are going to show that $\widetilde{\Ad(g)}-\id$ is nilpotent. This means that its unique eigenvalue is $0$, and since $\widetilde{\Ad(g)}$ is diagonalizable, we deduce that $\widetilde{\Ad(g)}=\id$, so that the permutation $\Ad(g)$ is trivial as required.  Clearly $(\Ad(g)-\id)|_{\h_K}$  is nilpotent because $\h_K$ is by definition the maximal subspace of the eigenvalue $1$. In particular, given $n\in \N$ big enough and any $\lambda\in \Lambda$ we have that
$$(\widetilde{\Ad(g)}-\id)^n(\lambda)=\sum_{k=0}^n\binom{n}{k} (-1)^{n-k}\widetilde{\Ad(g)}^k(\lambda)=$$
$$=\sum_{k=0}^n\binom{n}{k} (-1)^{n-k}\lambda \circ \Ad(g)^k=\lambda \circ \big(\sum_{k=0}^n\binom{n}{k} (-1)^{n-k} \circ \Ad(g)^k\big)=$$
$$=\lambda \circ [(\Ad(g)-\id)|_{\h_K}]^n=0,$$
so $\widetilde{\Ad(g)}-\id$ is nilpotent.
\end{proof}

\section{Regular points}

We fix  a definably connected definable group  $G$, an {$n$-dimensional}  $F$-vector space $V$, where $F=R$ or  $R(i)$ as usual, and  a \emph{definable representation}  of $G$ on $V$, \emph{i.e.} a definable continuous  homomorphism $\rho:G\rightarrow \GL(V)$  \emph{e.g.},\,the  adjoint representation $\Ad:G\rightarrow \GL(\mathfrak{g})$.  
 
 For each $g\in G$ consider $$det(\rho(g)-(T+1)\id)=a_0(g)+\cdots+a_{n-1}(g)T^{n-1}+T^n.$$
The functions $a_i:G\rightarrow R$ are definable  and continuous.  Let $r$ and $r^0$ be the maps from $G$ to $\N$ defined as follows. For each $g\in G$,
$$r(g):=\min\{j:a_j(g)\neq  0\}\qquad \text{and}\qquad r^0(g):=\min\{j:(a_j)_g\neq  0\},$$
where $(a_j)_g$ denotes the germ of $a_j$ at $g$.

We recall  the following notation. For each $\lambda\in F$, 
$$V^\lambda(\rho(g)):=\{X\in V:  (\rho(g)-\lambda \id)^n(X)=0\}$$
where $n=\dim V$.

\begin{remark}\label{remarks:r} Let $g\in G$.

(1) $\dim V^\lambda(\rho(g))$ is the multiplicity of $\lambda$ as a root of the characteristic polynomial of $\rho(g)$. In particular, we have that $r(g)=\dim(V^1(\rho(g)))$, which a priori it can be zero.

(2)  There is an open neighbourhood $U\sub G$  of $g$ such that $r(h)\leq r(g)$ for all $h\in U$. Indeed, if $r(g)=r$ then $a_r(g)\neq 0$ and therefore there is an open neighbourhood $U$ of $G$ such that $a_r(h)\neq 0$ for all $h\in U$. 

(3) $r(g)\geq r^0(g)$ for all $g\in G$ because $a_j(g)\neq 0$ implies that $(a_j)_g\neq 0$.
\end{remark}

\begin{definition}
An element $g\in G$ is called \emph{regular  with respect to}  $\rho$ if $r(g)=r^0(g)$.
The set of regular points is denoted by $\Reg_\rho(G)$, which is clearly definable. We say \emph{regular point of} $G$ if $\rho$ is $\Ad$, and then we just say write $\Reg(G)$.
\end{definition}

\begin{remark}The above definition is the natural adaptation to our context of \cite[Ch.VII,\S4, Def.1]{BourbLALG7-9}, the standard reference in the literature concerning regular points. In the case of Lie groups, the coefficients are analytic and therefore if $k$ is the minimum such that $a_k\neq 0$ then  $g$ is regular if and only if $a_k(g)\neq 0$. In particular, $\Reg(G)$ is a dense subset of $G$ and $r$ is constant on $\Reg(G)$. Then, working with the exponential map, it is possible to prove that the constant value of $r$ in $\Reg(G)$ is equal to $\rk \,\g$. In fact, $g\in \Reg(G)$ if and only if $\g^1(Ad(g))$ is a Cartan subalgebra of $\g$.  

To prove in the o-minimal context that $\Reg(G)$ is an open dense subset of $G$ is relatively easy, even though we do not know a priori if $r$ is constant on $\Reg(G)$ (Proposition \ref{reglarge}). It becomes more difficult to prove that  $g\in \Reg(G)$ if and only if $r(g)=\rk\, \g$ if and only if $\g^1(Ad(g))$ is a Cartan subalgebra of $\g$ (Theorem \ref{thm:Cartangenerla}).  The proof is again a reduction to the linear algebraic case, passing through the Zariski
closure $G_{1}$ of $\Ad(G)$ inside $\GL(\g_{K})$. To carry out this
reduction we first need to study the behaviour of $\g^1(Ad(g))$  under exact
sequences of definable groups (see Proposition \ref{prop:imageregular}), so we can control it under the adjoint representation. In particular, we will be able to relate the regular points of both $G$ and $G_1$. Once we have proved that if $g\in \Reg(G)$ then $\h:=\g^{1}(\Ad(g))$ is a Cartan subalgebra,  thanks to the main result in Section \ref{sect:CharCar}, we deduce that $C(\h)$ is the unique Cartan subgroup containing $g$ (Proposition \ref{prop:regequivalences}).

\end{remark}

We begin by proving some basic facts about regular points of a definable representation of $G$.


\begin{proposition}\label{prop:reg}Let $g\in G$. Then, $g\in \Reg_\rho(G)$ if and only if $r$ is constant in a neighbourhood of $g$. 
\end{proposition}
\begin{proof}Assume that $g$ is regular and denote $r=r(g)=r^0(g)$. Since $a_r(g)\neq 0$ and $a_j$ is continuous there is a neighbourhood $U$ of $g$ such that $a_r(h)\neq 0$, for all $h\in U$. On the other hand, since $r^0(g)=r$, we know that 
$(a_j)_g=0$, for all $j<r$. So we can assume that $(a_j)_h=0$, for all $h\in U$. This shows that $r^0(h)=r(h)=r$ for all $h\in U$, as required. 
 
Conversely, assume that $r(h)=r$ for all $h$ in a neighbourhood $U$ of $g$. Then, by definition that means that $a_r(h)\neq 0$ and $a_j(h)=0$ for all $h\in U$ and $j<r$, so $(a_j)_h\neq 0$ and $(a_j)_h=0$ for all $h\in U$, so that $r^0(h)=r$ for all $h\in U$. In particular, $g$ (and all the points in $U$) are regular points.
\end{proof}

\begin{proposition}\label{reglarge}$\Reg_\rho(G)$ is an open and dense subset of $G$. 
\end{proposition}
\begin{proof}$\Reg_\rho(G)$ is open in $G$ because of Proposition \ref{prop:reg}. Let us see that it is dense in $G$. Take $\mathcal{C}$ a cellular decomposition of $G$ compatible with the definable sets $X_i=\{g\in G : r(g)=i\}$ for $i=0,\ldots,n-1$. Clearly $G=X_0\cup \cdots \cup  X_{n-1}$. Let $X$ be the union of the cells of maximal dimension $\dim G$.  Clearly $X$ is open and dense in $G$. On the other hand, if $g\in X$ then $g$ belongs to a cell $C$ of dimension $\dim G$ with $C\subseteq X_i$ for some $i$. In particular, $r$ is constant (equal to $i$) in the open subset $C$ of $G$, thus  $g$ is regular by Proposition\,\ref{prop:reg}. 
\end{proof}

We define the \emph{rank} of a definable group $G$ as $$\rk(G):= \min\{\dim\g^1(\Ad(g)): g\in G\}.$$If $G$ is an irreducible algebraic subgroup of $\GL(n,F)$ we have two concepts of regular point of $G$. Let $\Reg^a(G)$ denote the set of regular points of $G$ in the sense of \cite[Ch.\,VI \S\,4 Def.\,2]{Chevalley55AlgLie}, that is, $$\Reg^a(G):=\{g\in G : r(g)=\rk(G)\}.$$ As we have mentioned in the proof of Theorem\,\ref{denselin}, $\Reg^a(G)$ is a nonempty Zariski open subset of $G$. In fact, we have the following.
\begin{prop}\label{reg^a} If $G$ is an irreducible algebraic subgroup of $\GL(n,F)$ then $\Reg^a(G)=\Reg(G)$.
\end{prop}
\proof Let $g\in \Reg^a(G)$. Since $\Reg^a(G)$ is open and $r$ is constant on  $\Reg^a(G)$, $r$ is locally constant around $g$, thus $g\in\Reg(G)$ by Proposition\,\ref{prop:reg}. Conversely, let $g\in\Reg(G)$ and take $U$ an open neighbourhood of $g$ such that $r$ is constant on $U$. Since $\Reg^a(G)$ is a Zariski open subset of $G$ it is dense in $G$ because $G$ is irreducible, and so $\Reg^a(G)\cap U\not=\emptyset.$ In particular, $r(g)=\rk(G)$. 
\endproof
\begin{cor}\label{regdense} Let $G$ be a definable subgroup of $\GL(n,F)$ and $\zar{G}{}$ its Zariski closure. Then, the set  $\Reg(\zar{G}{})\cap G$ is dense in $G$.
\end{cor}
\proof By Proposition\,\ref{reg^a} and the claim in the proof of Theorem\,\ref{denselin}.
\endproof


The next result will be useful when we consider quotients.
\begin{prop}\label{prop:imageregular}Let  $H$ be a definable group and let  $f:G\rightarrow H$ be a surjective definable homomorphism. Then,

$(1)$ for every $g\in G$ we have the following exact sequence
$$0\rightarrow \mathfrak{K}^1 (\Ad |_\mathfrak{K}(g) )\rightarrow \g^1 (\Ad(g) ) \rightarrow  \h^1 ((\Ad_H\circ f)(g) ) \rightarrow 0,$$
where  $\mathfrak{K}:=\Lie(\Ker(f))$, $\Ad:=\Ad_G$ and $\Ad_H$ are the adjoint representations of $G$ and $H$ respectively, and $\Ad |_\mathfrak{K}:G\rightarrow \GL(\mathfrak{K})$, and 

$(2)$ $f(\Reg (G))\subseteq \Reg (H)$.
\end{prop}
\begin{proof} 
(1) Consider the exact sequence of Lie algebras 
{$$0\rightarrow \mathfrak{K}_K\rightarrow \mathfrak{g}_K\rightarrow \mathfrak{h}_K\rightarrow 0.$$}
Let $\lambda\in K:=R(i)$. Firstly note that for all $g\in G$ and $Y\in \g_K$ we have that  $(\Ad_H\circ f)(g)(d_ef(Y))=d_ef(\Ad(g)(Y))$ and by induction we deduce that 
 $$[(\Ad_H\circ f)(g)-\lambda \id]^\ell(d_ef(Y))=d_ef\big([\Ad(g)-\lambda \id]^\ell(Y)\big), \text{ for each }\ell\in\N.$$
This shows that $d_ef(\g_K^\lambda(\Ad(g)))\subseteq \h_K^\lambda  ((\Ad_H\circ f)(g) ).$ Since $f$ is surjective and  $\g_K= \bigoplus_{\lambda\in K} \g_K^\lambda(\Ad(g))$, we deduce
$$\h_K=d_ef(\g_K)= \bigoplus_{\lambda\in K} d_ef\big(\g^\lambda_K(\Ad(g))\big)\subseteq \bigoplus_{\lambda\in K} \h_K^\lambda  ((\Ad_H\circ f)(g) )\subseteq \h_K$$
Thus, $d_ef(\g_K^\lambda(\Ad(g)))= \h_K^\lambda  ((\Ad_H\circ f)(g) )$ and hence, for each {$\lambda\in R$},
$d_ef(\g^\lambda(\Ad(g)))= \h^\lambda  ((\Ad_H\circ f)(g) ).$
 In particular, 
$$d_ef(\g^1(\Ad(g)))= \h^1  ((\Ad_H\circ f)(g) ).$$
Finally, the computation of the kernel is easy since $d_ef(Y)=0$ implies that $Y\in \mathfrak{K}\cap\g^1(\Ad(g))=\mathfrak{K}^1(\Ad|_{\mathfrak{K}}(g))$.

(2)  From (1) we have that, for every $g\in G$, 
$$\dim\g^1(\Ad(g))=\dim\mathfrak{K}^1(\Ad |_{\mathfrak{K}}(g))+ \dim\h^1((\Ad_H\circ f)(g)).$$
Therefore, if we denote by $r,r'$ and $r''$ the functions introduced at the beginning of this section of $\Ad, \Ad |_\mathfrak{K}$ and $\Ad_H\circ f$, respectively, then we have proved that
$r(g)=r'(g)+r''(g)$ 
for all $g\in G$.

Now, take $g\in \Reg (G):=\Reg _{\Ad}(G)$. By Proposition \ref{prop:reg}  there exists a neighbourhood $V$ of $g$ such that $r$ is constant in $V$. Moreover, by Remark \ref{remarks:r}(2) we can assume that for all $y\in V$ we have $r'(y)\leq r'(g)$ and $r''(y)\leq r''(g)$. In particular, for all $y\in V$ since
$$r'(y)+r''(y)=r(y)=r(g)=r'(g)+r''(g)$$
and $r'(y)\leq r'(g)$ and $r''(y)\leq r''(g)$, we have that $r'(y)=r'(g)$ and $r''(y)= r''(g)$. So we have that $r''$ is also constant in $V$. 

Finally, since $f$ is a definable homomorphism then it is an open map, so that $f(V)$ is an open neighbourhood of $f(g)$. Moreover, for each $z\in f(V)$ we have that $z=f(y)$ for some $y\in V$ so that
$$\dim\h^1(\Ad_H(z))=\dim\h^1 ((\Ad_H\circ f)(y))=r''(y)$$
so that $\dim\h^1 (\Ad_H(z))$ is constant, for all $z\in f(V)$. Thus, $f(g)\in \Reg (H)$.
\end{proof}

We have also the following corollaries of the above proposition, the first one  will allow us to transfer results from the  linear case to the general definable case.
\begin{cor}\label{cor:imageregular}Let  $H$ be a definable group and let  $f:G\rightarrow H$ be a surjective definable homomorphism with $\Ker(f)=Z(G)$. Then, given $g\in G$ we have 
the following exact sequence
$$0\rightarrow \mathfrak{z}(\g)\rightarrow \g^1(\Ad(g) ) \rightarrow  \h^1 ((\Ad_H\circ f)(g) ) \rightarrow 0.$$
Moreover, $g\in \Reg(G)$ if and only if $f(g)\in  \Reg (H)$.
\end{cor}
\begin{proof} {The exact sequence follows from (1) of Proposition \ref{prop:imageregular} noting that since $\Ker(f)=Z(G)$ we have that
$\mathfrak{K}^1(\Ad |_{\mathfrak{K}}(g))=\mathfrak{z}(\g)$.} Note that in particular,
$$\dim(\g^1(\Ad(g)))=\dim\mathfrak{z}(\g)+\dim\h^1((\Ad_H\circ f)(g)).$$
for all $g\in G$. 

For the last clause, by Proposition\,\ref{prop:imageregular}, it remains to prove that if $f(g)\in \Reg (H)$ then $g\in \Reg (G)$. 
In particular, since $f(g)\in \Reg (H)$ then there is an open neighbourhood $W$ of $f(g)$ in $H$ such that for all $z\in W$ we have that $\dim \h^1 (\Ad_H(z))=\dim\h^1 (\Ad_H(f(g)))$. Thus, for every $y \in f^{-1}(W)$ we have that 
$\dim\g^1(\Ad(y))=\dim\mathfrak{z}(\g)+\dim\h^1((\Ad_H\circ f)(y))$ $=\dim\mathfrak{z}(\g)+\dim\h^1(\Ad_H(f(g)))$
so that $\dim\g^1(\Ad(y))$ is constant for every $y$ in the open neighbourhood $f^{-1}(W)$ of $g$, that is, $g\in \Reg (G)$.
\end{proof}
\begin{cor}\label{esg/h} Let $H$ be a definable normal subgroup of $G$. Then, for every $g\in H$ we have the following exact sequence
$$0\rightarrow \h^1 (\Ad(g) )\rightarrow \g^1 (\Ad(g) ) \rightarrow \g/\h\rightarrow 0.$$
\end{cor}
\proof We apply Proposition \ref{prop:imageregular}  to the projection $\pi:G\to G/H$ and obtain, for each $g\in H$, the exact sequence
$$0\rightarrow \h^1 (\Ad |_{\h}(g) )\rightarrow \g^1 (\Ad(g) ) \rightarrow (\g/\h)^1 (\Ad_{G/H}\pi(g) )\rightarrow 0.$$ We conclude noting that since $g\in H$ we have $\h^1 (\Ad |_{\h}(g) )=\h^1 (\Ad(g) )$ and {$(\g/\h)^1 (\Ad_{G/H}\pi(g))=(\g/\h)^1 (\id_{\g/\h})=\g/\h$}.
\endproof

Next result will allow us to prove  that a regular point of $G$ belongs to  a unique Cartan subgroup of $G$. 
\begin{lemma}\label{lemma:ginCh} Let $\h$ be a Cartan subalgebra of $\g$.  Then $\h \sub\g^1(\Ad(g))$, for any $g\in C(\h)$. 
\end{lemma}
\begin{proof} Let  $g\in C(\h)$. Since  $\z(\g)$ belongs to both $\h$ and $\g^1(\Ad(g))$, to prove that $\h\subseteq \g^1(\Ad(g))$ it suffices to show that $\ad(\h)\sub \ad(\g^1(\Ad(g)))$, or equivalently, that $\n:=\ad(\h_K)$ is contained in $\ad(\g_K^1(\Ad(g)))$. Let $G_1:=\zar{\Ad(G)}{K}$ (as in \S\,4).  Note that  $\Lie(\Ad(G))=\ad(\g)$ and
 $$\Ad_{\Ad(G)}(h)=\Ad_{G_1}(h)|_{\ad(\g)} \text{ for each } h\in \Ad(G),$$ so we apply Corollary  \ref{cor:imageregular} to the map $\Ad:G\to \Ad(G)$ and,  after complexifying, obtain
$$0\rightarrow \mathfrak{z}(\g_K) \rightarrow \g_K^1 (\Ad(g) ) \xrightarrow{\ad}  \ad(\g_K)^1 ((\Ad_{G_1}\circ \Ad)(g) ) \rightarrow 0.$$
In particular, since $\ad(\g_K)\sub\g_1$, we get
$$\ad(\g^1_K(\Ad(g)))=\ad(\g_K)^1 (\Ad_{G_1}(\Ad(g)) )= \g^1_1(\Ad_{G_1}(\Ad(g)))\cap \ad(\g_K).$$
Since  $\n$ is contained in $\ad(\g_K)$ it only remains to prove that $$\n\sub \g^1_1(\Ad_{G_1}(\Ad(g))).$$

Let $H_1:=\zar{\Ad(C(\h))}{K}$. We claim that $H_1=Z_{G_1}(\n_s)$.  Indeed, first note that  by Corollary\,\ref{centralizerCartansubgroup} $Z_{G_1}(\n_s)$ is a Cartan subgroup of $G_1$, hence algebraic and irreducible. Thus, it suffices to prove that $\Ad(C(\h))\sub Z_{G_1}(\n_s)$ and $\Lie(H_1)=\z_{\g_1}(\n_s)$. The inclusion is obtained by Fact\,\ref{NCartan} (see proof of Theorem\,\ref{Cartan=C-Cartan}). On the other hand, we have  $$\h_1=a(\Lie(\Ad(C(\h)))=a(\ad(\h))=\z_{\g_1}(\n_s)$$by Proposition\,\ref{zarcl}, Lemma\,\ref{Lie-of-C-h} and Proposition\,\ref{centraliserCSA}, respectively. 

Next, note that, since $a(\n)$ is a nilpotent subalgebra of $\g_1$, each of the endomorphisms in $\ad_{a(\n)}(a(\n))$ is nilpotent.  We also have that $\ad_{a(\n)}(a(\n))$  is the Lie algebra of   $\Ad_{H_1}(H_1)$ which  is an irreducible algebraic group because $H_1$ is so.
 Therefore,   by \cite[Ch\,V \S\,3 Prop.\,14]{Chevalley55AlgLie} we deduce that $\Ad_{H_1}(h)-\id=[\Ad_{G_1}(h) -\id]_{|\h_1}$ is nilpotent, for every $h\in H_1$. In particular, since  $\Ad(C(\h))\sub H_1$ the endomorphism $[\Ad_{G_1}(\Ad(g))-\id]_{|\h_1}:\h_1 \rightarrow \h_1 $ is nilpotent and therefore $\h_1\sub\g_1^1(\Ad_{G_1}(\Ad(g)))$ and a fortiori $\n\sub\g_1^1(\Ad_{G_1}(\Ad(g)))$.
\end{proof}

\begin{prop}\label{thm:regareCartan} Let $G$ be a definably connected subgroup of $\GL(n,F)$ and $\zar{G}{}$ its Zariski closure. For any  $g\in \Reg (\zar{G}{})\cap G$,   $\g^1(\Ad(g))$ is a Cartan subalgebra of $\g$ and $g$ belongs to the Cartan subgroup $C(\g^1(\Ad(g)))$ of $G$.
\end{prop}
\begin{proof} Let $G_1:=\zar{G}{}$. Let  $g\in \Reg (G_1)\cap G$. By Proposition\,\ref{prop:reginCartan} it suffices to prove  that $\g^1(\Ad(g))$ is a Cartan subalgebra of $\g$. To do this  firstly note that   $\g_1=a(\g)$ and $\g^1(\Ad(g))= \g_1^1(\Ad(g))\cap\g$. Hence, by \cite[Ch\,VI \S\,4 Prop.\,21]{Chevalley55AlgLie}  will be enough to prove that $\g_1^1(\Ad(g))$ is a Cartan subalgebra of $\g_1$. The latter  is classical, here are the details.    By  [\emph{Ibid.} Prop.\,13]  $\g_1^1(\Ad(g)) \supseteq \h_1$ for some Cartan subalgebra $\h_1$ of $\g_1$.  Now, $g\in\Reg(G_1)$ means $\dim\g_1^1(\Ad(g))=\rk\,\g_1$. On the other hand $\h_1$ being a Cartan subalgebra of $\g_1$ has dimension equal to $\rk\,\g_1$, hence $\g_1^1(\Ad(g))=\h_1$.
%
\end{proof}

\begin{prop}\label{prop:regequivalences}Let $G$ be a definably connected subgroup of $\GL(n,F)$ and  $\zar{G}{}$  its Zariski closure. For any $g\in G$, the following conditions are equivalent.
\begin{itemize}
\item[{0)}] $g\in  \Reg (\zar{G}{})$;
\item[{1)}] $g\in  \Reg (G)$;
\item[{2)}] $\dim\g^1(\Ad(g))=\rk\,\g$, and 
\item[{3)}] $\g^1(\Ad(g))$ is a Cartan subalgebra of $\g$.
\end{itemize}
\end{prop} 
\proof
0) implies 3) follows from Proposition\,\ref{thm:regareCartan}, and  by definition of rank of a Lie algebra we have 
3) implies 2).

0) implies 1): Let $g\in\Reg(\zar{G}{})\cap G$. Let $U\sub G$  open such that $g\in U$ and $\dim\g^1(\Ad(h))\leq \dim\g^1(\Ad(g)) $, for all $h\in U$. Suppose that there is $h\in U$ such that $\dim\g^1(\Ad(h))< \dim\g^1(\Ad(g))$. {Then, there is an open subset $V$ of $U$ such that for all $h\in V$, $\dim\g^1(\Ad(h))< \dim\g^1(\Ad(g))$.  On the other hand, Corollary\,\ref{regdense} implies that there is $h\in V\cap Reg(\zar{G}{})\cap G$, and by 2)  $\g^1(\Ad(h))$ and $\g^1(\Ad(g))$ have equal dimension ($=\rk\,\g$), a contradiction.}
  
1) implies 0): Let $g\in\Reg(G)$. By  Proposition\,\ref{prop:reg} there is an open  neighbourhood $U\sub G$ of $g$ such that for all $h\in U$, we have that 
$\dim\g^1(\Ad(g))=\dim\g^1(\Ad(h))$. On the other hand,  by Proposition\,\ref{zarcl}  $G$ is normal in $\zar{G}{}$ so we can apply Corollary\,\ref{esg/h} and get that for any $g'\in G$   the exact sequence $$0\rightarrow \g^1 (\Ad(g') )\rightarrow \g_1^1 (\Ad(g') ) \rightarrow \g_1/\g\rightarrow 0,$$where $\g_1:=\Lie(\zar{G}{})$. Thus,  $\dim\g_1^1 (\Ad(g') )=\dim\g^1 (\Ad(g') ) + \dim\g_1/\g$, for any $g'\in G$. In particular, for every $h\in U$, $\dim\g_1^1 (\Ad(g) )=\dim\g_1^1 (\Ad(h) )$.  Now, Corollary\,\ref{regdense} implies that there is $h\in U\cap Reg(\zar{G}{})$ and so, for this $h$ we have $\dim\g_1^1 (\Ad(h) )=\rk\,\g_1$. Hence, $g\in\Reg(\zar{G}{})$.

2) implies 0): Let $g\in G$ such that $\dim\g^1(\Ad(g))=\rk\,\g$. Then,  $$\dim\g_1^1(\Ad(g))= \dim\g^1(\Ad(g))+ \dim\g_1/\g= \rk\,\g + \dim\g_1/\g=\rk\,\g_1,$$the first equality as above by Proposition\,\ref{zarcl} and Corollary\,\ref{esg/h}, the last one by Corollary\,\ref{rka(g)}.
\endproof

We recall that  if $\mathfrak a$ is a Lie algebra, $\z$ its centre and $\h$ a vector subspace of $\mathfrak a$ then, $\h$ is a Cartan subalgebra of $\mathfrak a$ if and only if  $\h$ contains $\z$ and $\h/\z$ is a Cartan subalgebra of $\mathfrak a/\z$. In particular, $\rk\,\mathfrak a=\rk(\mathfrak{a}/\z)+\dim\mathfrak{z}$.


\begin{thm}\label{thm:Cartangenerla}Let $G$ be a definably connected group. For any  $g\in G$, the following conditions are equivalent.
\begin{itemize}
\item[{1)}] $g\in  \Reg (G)$;
\item[{2)}] $\dim\g^1(\Ad(g))=\rk\,\g$, and 
\item[{3)}] $\g^1(\Ad(g))$ is a Cartan subalgebra of $\g$.
\end{itemize}
\end{thm} 
\begin{proof}Let $f:G\rightarrow H\leq \GL(n,R)$ be the map $f(g)=\Ad_G(g)$, for any $g\in G$, and $H:= \Ad_G(G)$. Note that $\Ker(f)=Z(G)$ and   $\rk\,\g=\rk\,\h+\dim\mathfrak{z}(\g)$.

1) implies 2): Let  $g\in  \Reg (G)$. By Corollary \ref{cor:imageregular} we have that $f(g)\in  \Reg (H)$. By Proposition\,\ref{prop:regequivalences} we have that $\dim\h^1(\Ad_H(f(g)))=\rk\,\h$. By Corollary \ref{cor:imageregular} again we have that
$$\dim\g^1(\Ad_G(g))=\dim\mathfrak{z}(\g)+\dim\h^1(\Ad_H(f(g)))=\dim\mathfrak{z}(\g)+\rk\,\h=\rk\,\g$$
as required.

2) implies 3): Take $g\in G$ with $\dim\g^1(\Ad(g))=\rk\,\g$. By Corollary\,\ref{cor:imageregular} we have that
$$\dim\h^1(\Ad_H(f(g))=\dim\g^1(\Ad_G(g))-\dim\mathfrak{z}(\g)=\rk\,\g-\dim\mathfrak{z}(\g)=\rk\,\h,$$
so that by Proposition\,\ref{prop:regequivalences} we have that $\h^1(\Ad_H(f(g))$ is a Cartan subalgebra of $\h$.
Since $\h^1(\Ad_H(f(g))=\ad(\g^1(\Ad(g)))$ because of Corollary \ref{cor:imageregular}, we deduce that $\g^1(\Ad(g))$ is a Cartan subalgebra of $\g$.

3) implies 1): Let $g\in G$ with $\g^1(\Ad(g))$ a Cartan subalgebra of $\g$. By Corollary \ref{cor:imageregular} we have that $\h^1(\Ad_H(f(g))=\ad(\g^1(\Ad_G(g)))$ and therefore  $\h^1(\Ad(f(g))$ is a Cartan subalgebra of $\h$. Thus, by Proposition\,\ref{prop:regequivalences}  $f(g)$ is regular element of $H$ and so $g$ is regular in $G$ by Corollary \ref{cor:imageregular}. 
\end{proof}

\begin{cor}\label{cor:reguniqueCar}Let $G$ be a definably connected group.  Then, $\Reg(G)$ is a dense subset of $G$, and if $g\in \Reg (G)$ then $g$ belongs to a unique Cartan subgroup of $G$.
\end{cor}
\begin{proof} The first statement  is by Proposition\,\ref{reglarge}. Also, we already know that $\g^1(\Ad(g))$ is Cartan subalgebra of $\g$ and thus $C(\g^1(\Ad(g)))$ is a Cartan subgroup containing $g$ by Proposition\,\ref{prop:reginCartan}. If $H$ is another Cartan subgroup of $G$ with $g\in H$, then $H=C(\h)$, where $\h$ is the Lie algebra of $H$, by Theorem\,\ref{Cartan=C-Cartan}. By Lemma \ref{lemma:ginCh} we have that $\h\sub\g^1(\Ad(g))$, and since both $\h$ and $\g^1(\Ad(g))$ are Cartan subalgebras of $\g$, we get $\h= \g^1(\Ad(g))$. So that $H=C(\h)=C(\g^1(\Ad(g)))$, as required. 
\end{proof}

\begin{rem} Just for the record, we write an alternative proof of the above corollary in the linear case (the general case can be deduced easily from this).  Let $G$ be a definable subgroup of $\GL(n,F)$  and $G_1:=\zar{G}{}$. Let $g\in \Reg (G)$ and let $H$ and $\widetilde{H}$ be Cartan subgroups of $G$ with $g\in H\cap \widetilde{H}$. Let $\h$ and $\widetilde{\h}$ be the Lie algebras of $H$ and $\widetilde{H}$ respectively, which are Cartan subalgebras of $\g$.  Let  $\h_1=a(\h)$ and $\widetilde{\h}_1=a(\widetilde{\h})$.  By  \cite[Ch.VI \S\,4 Prop.21]{Chevalley55AlgLie} we have that $\h_1$ and $\widetilde{\h}_1$ are Cartan subalgebras of $\g_1$. Let $H_1$ and $\widetilde{H}_1$ be the Cartan subgroups whose Lie algebra are $\h_1$ and $\widetilde{\h}_1$ respectively (see \cite[Ch.VI \S\,4 Prop.5]{Chevalley55AlgLie}).

On the other hand, take the Zariski closure $\zar{H}{}$ of $H$ in $G_1$. Clearly, the Lie algebra $\h_1$ is contained in the Lie algebra of $\zar{H}{}$, and therefore, since $H_1$ is irreducible \cite[Ch.VII \S\,4 Thm.2]{BourbLALG7-9} we have that $H_1$ is a subgroup of $\zar{H}{}$. But since $H$ is nilpotent we have that $\zar{H}{}$ is nilpotent, so by maximality of $H_1$ we get that 
$H_1=\zar{H}{}$. By a similar argument, $\widetilde{H}_1=\zar{\widetilde{H}}{}$. 

Now, by Proposition\,\ref{prop:regequivalences} we have that $g\in \Reg (G_1)$, so that $g$ belongs to a unique Cartan subgroup (see \cite[Ch.VI \S\,4 Thm.2]{BourbLALG7-9}). Thus, we have that $H_1=\widetilde{H}_1$, and therefore $\h_1=\widetilde{\h}_1$. In particular, the Cartan algebras $\h=\h_1\cap \g$ and $\widetilde{\h}=\widetilde{\h}_1\cap \g$ are equal, so that $\h=\widetilde{\h}$. Finally, $H=C(\h)=C(\widetilde{\h})=\widetilde{H}$, as required.

\end{rem}

\bibliographystyle{plain}
\bibliography{biblio}
\end{document}